\documentclass[12pt]{amsart}
\usepackage[dvips]{color}
\usepackage{amsmath}
\usepackage{amsxtra}
\usepackage{amscd}
\usepackage{amsthm}
\usepackage{amsfonts}
\usepackage{amssymb}
\usepackage{eucal}
\usepackage{epsfig}
\usepackage{graphics}
\usepackage{accents}

\usepackage[all]{xy}
\numberwithin{equation}{section}
\newtheorem{thm}{Theorem}[section]
\newtheorem{prop}[thm]{Proposition}
\newtheorem{lem}[thm]{Lemma}

\newtheorem{cor}[thm]{Corollary}
\newtheorem{conj}[thm]{Conjecture}

\newcommand{\nn}{\nonumber}

\newcommand{\bra}[1]{\langle #1 |}        
\newcommand{\ket}[1]{{| #1 \rangle}}      
\newcommand{\ds}[1]{\displaystyle #1}

\newcommand{\C}{{\mathbb C}}
\newcommand{\Z}{{\mathbb Z}}

\newcommand{\B}{{\mathcal B}}
\newcommand{\E}{{\mathcal E}}
\newcommand{\F}{\mathcal F}
\newcommand{\cH}{\mathcal H}

\newcommand{\N}{{\mathcal N}}

\newcommand{\cP}{\mathcal{P}}
\newcommand{\cR}{\mathcal{R}}

\newcommand{\ab}{\mathbf{\mathfrak{a}}}

\newcommand{\bs}{\boldsymbol}

\newcommand{\GS}{\mathfrak{S}}
\newcommand{\gl}{\mathfrak{gl}}

\newcommand{\la}{\lambda}
\newcommand{\La}{\Lambda}

\newcommand{\bw}{\mathbf{w}}
\newcommand{\bz}{\mathbf{z}}

\newcommand{\ssF}{\textsf{F}}
\newcommand{\ssG}{\textsf{G}}
\newcommand{\ssI}{\textsf{I}}
\newcommand{\ssT}{\textsf{T}}

\newcommand{\id}{{\rm id}}

\newcommand{\sgn}{\mathop{\rm sgn}}
\newcommand{\Sym}{\mathrm{Sym}}
\newcommand{\Tr}{{\rm Tr}}

\newcommand{\Cp}{C^{\perp}}
\newcommand{\Dp}{D^{\perp}}

\newcommand{\mc}{\mathcal}
\newcommand{\al}{\alpha}

\newcommand{\cN}{\mathcal{N}}

\newcommand{\g}{\mathfrak{g}}

\newcommand{\bK}{\bar{K}}
\newcommand{\hE}{\hat{E}}
\newcommand{\hF}{\hat{F}}
\newcommand{\hK}{\hat{K}}

\begin{document}

\begin{title}[Integrals of motion from quantum toroidal algebras]
{Integrals of motion from quantum toroidal algebras}
\end{title}
\author{B. Feigin, M. Jimbo, and E. Mukhin}
\address{BF: National Research University Higher School of Economics, 
Russian Federation, International Laboratory of Representation Theory 
and \newline Mathematical Physics, Russia, Moscow,  101000,  
Myasnitskaya ul., 20 and Landau Institute for Theoretical Physics,
Russia, Chernogolovka, 142432, pr.Akademika Semenova, 1a
}
\email{bfeigin@gmail.com}
\address{MJ: Department of Mathematics,
Rikkyo University, Toshima-ku, Tokyo 171-8501, Japan}
\email{jimbomm@rikkyo.ac.jp}
\address{EM: Department of Mathematics,
Indiana University-Purdue University-Indianapolis,
402 N.Blackford St., LD 270,
Indianapolis, IN 46202, USA}\email{emukhin@iupui.edu}

\dedicatory{To the memory of Petr Kulish}
\begin{abstract} 
We identify the Taylor coefficients of the transfer matrices corresponding to quantum toroidal algebras
with the elliptic local and non-local integrals of motion
introduced by Kojima, Shiraishi, Watanabe, and one of the authors.

That allows us to prove the Litvinov conjectures on the Intermediate Long Wave model.

We also discuss the $({\mathfrak {gl}}_m,{\mathfrak {gl}}_n)$ duality of XXZ models in
quantum toroidal setting and the implications for the quantum KdV model. In particular, we conjecture that the spectrum 
of non-local integrals of motion 
of Bazhanov, Lukyanov, and Zamolodchikov 
is described by Gaudin Bethe ansatz equations 
associated to affine ${\mathfrak{sl}}_2$.

\end{abstract}
\date{\today}
\maketitle 
  
\section{Introduction}

The quantum KdV (qKdV) model was introduced in \cite{SY}, \cite{EY}. 
It is described by a remarkable commutative subalgebra of local Hamiltonians 
(local integrals of motion) in the completion of the universal enveloping algebra 
of the Virasoro algebra. 
The first nontrivial local Hamiltonian is 
$L_0^2+2\sum_{k>0}L_{-k}L_k$, where the $L_k$ are the standard Virasoro generators. 
The local Hamiltonians act on 
highest weight Virasoro modules, and the 
problem is to diagonalize these linear operators.

The study of  qKdV and similar models received a new boost after the 
groundbreaking series of papers \cite{BLZ1}-
\cite{BLZ5}. 
Apart from many other results, 
it was suggested that the spectrum of the model is described by the condition of the 
triviality of the monodromy of certain differential operators. Then the absence of 
monodromy is readily translated to a Bethe ansatz type equation.

In an attempt to understand the situation,  it was conjectured in \cite{L} that 
the algebra of local Hamiltonians can be naturally deformed at the expense of 
adding an extra Heisenberg algebra. The resulting Hamiltonians are called 
Intermediate Long Wave (ILW) Hamiltonians. 
They are quantization of a one-parameter family of 
integrable systems interpolating 
between the KdV and the Benjamin-Ono hierarchies.
The ILW Hamiltonians depend on 
a deformation parameter $\tau$, see \eqref{Litv-I2} for the first nontrivial 
ILW Hamiltonian. In the conformal limit $\tau\to 0$ the system splits and the 
quantum qKdV Hamiltonians are recovered.
It was conjectured in \cite{L} that the spectrum of ILW Hamiltonians is 
described by solutions of another, different looking Bethe ansatz equation, 
see \eqref{BAE limit}.

\medskip

The present work is an outcome of
an attempt to understand the Bethe ansatz answers 
predicted for the spectrum of the qKdV and ILW models. Our starting point is the Bethe 
ansatz for the XXZ model associated to quantum toroidal $\gl_1$ algebra, denoted by $\mc E_1$, 
which was developed in \cite{FJMM1} and \cite{FJMM2}. 

The ILW model
corresponds to the case of a tensor product of two Fock representations of $\mc E_1$.  
Using the results of \cite{FT}, we write the Taylor coefficients of $\mc E_1$ transfer matrices 
as certain integrals of products of generating currents, see Proposition \ref{gl1 transfer}.  
We use the vertex operator realization of the $\mc E_1$ action on the
Fock space, see \cite{Sa}, \cite{STU}, Section \ref{sub:Fock}. It turns out that in the tensor product of two Fock spaces, the Taylor 
coefficients of $\mc E_1$ transfer matrices  
are equivalent to the
elliptic local integrals of motion 
of \cite{FKSW}, see Corollary \ref{gl1 compare}. 
  
In the picture, we have  parameters $q_1,q_2$ of $\mc E_1$, the twisting parameter of the 
transfer matrix $p$, the ratio of evaluation parameters of two Fock spaces $u_1/u_2$. 
The scaling limit $q_2\to 1$, $q_1\to 1$ with $-\log q_1/\log q_2=r$ is fixed, is the 
Yangian limit.  We observe that in this limit, the first nontrivial coefficient of 
$\mc E_1$ transfer matrix gives rise to
the ILW Hamiltonian \eqref{Litv-I2}. 
The parameters $r$ and $u_1/u_2$ 
are related respectively 
to the central charge and the highest weight of 
the Virasoro module, and $p$ is related to
the deformation parameter $\tau$, see Lemma \ref{Litvinov}.

By \cite{FJMM1}, \cite{FJMM2}, the 
 spectrum of $\mc E_1$ transfer matrices is described by Bethe ansatz equations, see 
Proposition \ref{BA gl1 prop}. We show that the Yangian limit of $\mc E_1$ Bethe ansatz 
equations \eqref{BAE gl1} coincides with the ILW Bethe ansatz equations of \cite{L}, see 
Corollary \ref{ILW spectrum}. Thus we prove the conjectures of \cite{L}.

\medskip

The qKdV model has another remarkable set of commuting Hamiltonians, 
called non-local integrals of motion in
\cite{BLZ1}. The non-local integrals 
of motion commute with local Hamiltonians and therefore have the same eigenvectors. 
In this paper we propose a Bethe ansatz answer for the spectrum of non-local integrals of motion.

We identify the 
tensor product of two $\mc E_1$ Fock spaces 
with a subspace of fixed
momentum in a single Fock space of quantum toroidal $\gl_2$ algebra, denoted 
$\mc E_2$. For $\mc E_2$, we have parameters $q_1^\vee=q_2p^{-1/2}$,  $q_2^\vee=q_2$, 
the twisting parameters of the transfer matrices 
$p^\vee=q_1^{-1}q_2^{-1}$ 
and $p_1^\vee=u_2/u_1$.

Similarly to the case of $\mc E_1$, we express the 
Taylor coefficients of $\mc E_2$ transfer matrices as certain integrals of products 
of generating currents, see Proposition \ref{gl2 transfer}, and observe that 
on Fock modules
these integrals coincide with the
elliptic non-local integrals of motion of \cite{FKSW}, see 
Corollary \ref{gl2 compare}.

Then, remarkably, the $\mc E_1$ transfer matrices commute with $\mc E_2$ transfer matrices. 
This can be shown by a direct computation involving the simplest non-trivial Taylor coefficients.
It also follows from the identifications with \cite{FKSW} where the commutation of all local and non-local integrals 
of motion was explicitly checked.

We call this phenomenon the $(\mc E_1, \mc E_2)$ duality of XXZ models. 

We observe a similar  $(\mc E_1, \mc E_n)$ duality of XXZ models, cf. \cite{FKSW2}, \cite{KS}, and conjecture 
that there exists a  $(\mc E_m, \mc E_n)$ duality for all $m, n$, see Section \ref{duality sec}. 
This can be viewed as a quantum toroidal version of the $(\gl_m,\gl_n)$ duality of Gaudin systems, see 
\cite{MTV1}. In the affine setting, the new feature is the exchange of the twisting parameter 
corresponding to the null 
root with the parameter in the relations of the dual algebra. 
We are planning to describe the $(\mc E_m, \mc E_n)$ duality in detail in a future publication.

We expect that Bethe ansatz for $\E_n$ algebra can be established along the lines of 
\cite{FJMM1}, \cite{FJMM2}. The conjectural answer for $\mc E_2$ Bethe ansatz is given in 
Conjecture \ref{gl2 BAE conj}. Then we study the ILW and conformal limits of the $\mc E_2$ 
Bethe equations. The ILW limit of the Bethe equations looks new, see Lemma \ref{gl2 ILW}. 
In the conformal limit, the $\mc E_2$ Bethe ansatz equations turn into the
well-known Gaudin 
Bethe ansatz equations associated to affine $\mathfrak {sl}_2$ acting on
a tensor product of a 
Verma module with the basic level one module, see Lemma \ref{affine Gaudin}. The multiplicity 
space of this tensor product which contains the tensor product of highest weight vectors is 
naturally identified with the Virasoro Verma module with central charge and highest weight 
matching those 
of qKdV. Thus, we expect that the 
$\widehat{\mathfrak {sl}}_2$ Gaudin
equations describe the spectrum of non-local integrals of motion 
of \cite{BLZ1}. We conjecture the precise formula for the first nontrivial non-local integral, 
see \eqref{Q2 spectrum}.

Thus at the present moment, for the spectrum of qKdV we have 3 sets of Bethe equations: 
double Yangian $\gl_1$ (with $\tau=0$) equations as in \cite{L}, 
$\widehat{\mathfrak {sl}}_2$ Gaudin
equations, and the original equations of \cite{BLZ1}. Only the first set is proved (one still 
needs to study the completeness). The first and the third are related to local integrals of 
motion and the second one to non-local integrals of motion. It would be interesting to find a 
direct bijection between the solutions of the three systems.

\medskip

The text is organized as follows. After setting up the notation in Section \ref{sec:notation}, 
we write explicit formulas for Taylor coefficients of transfer matrices in Section \ref{sec:IM}. 
In Section \ref{sec:compare} we compare these formulas with elliptic local and non-local 
integrals 
of motion. In Section \ref{sec:BA}, we recall Bethe ansatz for $\mc E_1$ and
prove some technical results about the convergence of the Taylor coefficients of $\mc E_1$ 
transfer matrices. We also describe the conjectural answer for $\mc E_2$  Bethe ansatz. 
In Section \ref{secILW} we take the ILW limit and prove the Litvinov conjectures. 
We discuss the quantum toroidal $(\gl_m,\gl_n)$ duality of XXZ models and its consequences in Section 
\ref{sec:duality}. In the Appendix we prove an identity which is used in the proof of 
convergence of the Taylor coefficients of transfer matrices.

\section{Notation}\label{sec:notation}
In this section we set up our notation. 

\subsection{Quantum toroidal $\gl_n$}\label{sec:En}
Throughout the text we fix $q,d\in\C^\times$ and set 
$q_1=q^{-1}d, q_2=q^2, q_3=q^{-1}d^{-1}$ so that $q_1q_2q_3=1$. 
We assume that 
$q_1^kq_2^lq_3^m=1$ ($k,l,m\in\Z$) implies $k=l=m$.
Let $n$ be a positive integer. 
We denote by $(a_{i,j})_{i,j\in\Z/n\Z}$ the Cartan matrix of type 
$A^{(1)}_{n-1}$ (for $n=1$ we set $a_{0,0}=0$), and
by $(m_{i,j})_{i,j\in\Z/n\Z}$ the skew symmetric matrix 
whose only non-zero entries are $m_{i\pm1,i}=\pm1$. 

We denote by $\E_n$
the quantum toroidal algebra  of type $\gl_n$. Algebra $\E_n$ is generated by  
$E_{i,k},F_{i,k},H_{i,r}$ 
and invertible elements $K_{i}, C,C^{\perp},D,D^\perp$, 
where $i\in\Z/n\Z$, $k\in \Z$, $r\in\Z\backslash\{0\}$. 
Introduce the generating series
$E_i(z)=\sum_{k\in\Z}E_{i,k}z^{-k}$,
$F_i(z)=\sum_{k\in\Z}F_{i,k}z^{-k}$,
$K^{\pm}_i(z)=K_i^{\pm1}\bar{K}^\pm_i(z)$, 
$\bar{K}^\pm_i(z)=\exp\bigl(\pm(q-q^{-1})\sum_{r>0}H_{i,\pm r}z^{\mp r}
\bigr)$. 
Then the defining relations read as follows.
\begin{gather*}
\text{$C,\ \Cp$ are central},\quad
\Cp=K_{0}K_{1}\cdots K_{n-1}\,,
\\
D E_i(z) D^{-1}= E_i(q z)\,,\ 
D F_i(z) D^{-1}= F_i(q z)\,,\ 
D K^\pm_i(z) D^{-1}= K^{\pm}_i(q z)\,,
\\ 
\Dp E_i(z) (\Dp)^{-1}=q^{-\delta_{i,0}} E_i(z)\,,\ 
\Dp F_i(z) (\Dp)^{-1}=q^{\delta_{i,0}} F_i(z)\,,\ 
\Dp K^\pm_i(z) (\Dp)^{-1}= K^{\pm}_i(z)\,,
\end{gather*}
\begin{gather*}
K^\pm_i(z)K^\pm_j (w) = K^\pm_j(w)K^\pm_i (z), 
\\
\frac{g_{i,j}(C^{-1}z,w)}{g_{i,j}(Cz,w)}
K^-_i(z)K^+_j (w) 
=
\frac{g_{j,i}(w,C^{-1}z)}{g_{j,i}(w,Cz)}
K^+_j(w)K^-_i (z),
\\
d_{i,j}g_{i,j}(z,w)K_i^\pm(C^{-(1\pm1)/2}z)E_j(w)+g_{j,i}(w,z)E_j(w)K_i^\pm(C^{-(1\pm1) /2}z)=0,
\\
d_{j,i}g_{j,i}(w,z)K_i^\pm(C^{-(1\mp1)/2}z)F_j(w)+g_{i,j}(z,w)F_j(w)K_i^\pm(C^{-(1\mp1) /2}z)=0\,,
\end{gather*}
\begin{gather*}
[E_i(z),F_j(w)]=\frac{\delta_{i,j}}{q-q^{-1}}
(\delta\bigl(C\frac{w}{z}\bigr)K_i^+(w)
-\delta\bigl(C\frac{z}{w}\bigr)K_i^-(z)),\\
d_{i,j}g_{i,j}(z,w)E_i(z)E_j(w)+g_{j,i}(w,z)E_j(w)E_i(z)=0, \\
d_{j,i}g_{j,i}(w,z)F_i(z)F_j(w)+g_{i,j}(z,w)F_j(w)F_i(z)=0.\\
\end{gather*}
Here we set
\begin{align*}
n\ge 3&:\quad
g_{i,j}(z,w)=\begin{cases}
	      z-q_1w & (i\equiv j-1),\\
              z-q_2w & (i\equiv j),\\
	      z-q_3w & (i\equiv j+1),\\
              z-w & (i\not\equiv j,j\pm1),\\
	     \end{cases}\\
n=2&:\quad 
 g_{i,j}(z,w)=\begin{cases}
	      z-q_2w & (i\equiv j),\\
              (z-q_1w)(z-q_3w)& (i\not\equiv j),
	     \end{cases}\\
n=1&:\quad 
 g_{0,0}(z,w)=(z-q_1w)(z-q_2w)(z-q_3w),
\end{align*}
and 
$d_{i,j}=d^{\mp 1}$ ($i\equiv j\mp1, n\ge 3$), $=-1$ ($i\not\equiv j, n=2$), 
$=1$ (otherwise). 

We have 
\begin{align*}
K_i E_j(z) K_i^{-1}=q^{a_{i,j}}E_j(z)\,,\
K_i F_j(z) K_i^{-1}=q^{-a_{i,j}}F_j(z)\,.
\end{align*}
For $r\neq 0$, 
\begin{align*}
&[H_{i,r},E_j(z)]= a_{i,j}(r)C^{-(r+|r|)/2}  
\,z^r E_j(z)\,,\\
&[H_{i,r},F_j(z)]=-a_{i,j}(r)C^{-(r-|r|)/2}   
\,z^r F_j(z)\,,\\
&[H_{i,r},H_{j,s}]=\delta_{r+s,0} a_{i,j}(r) \frac{C^r-C^{-r}}{q-q^{-1}}
 \,,
\end{align*}
where  $[x]=(q^x-q^{-x})/(q-q^{-1})$ and 
\begin{align*}
 a_{i,j}(r)
=\frac{1}{r}\times
\begin{cases}
[r a_{i,j}]d^{-r m_{i,j}}&(n\ge 3)\,,\\
[2r]\delta_{i,j}-[r](d^r+d^{-r})(1-\delta_{i,j})& (n=2)\,,\\
-\frac{1}{q-q^{-1}}(1-q_1^r)(1-q_2^r)(1-q_3^r)& (n=1)\,.\\
\end{cases}
\end{align*}
In addition there are Serre relations.
Since they will not be used in this paper, 
we omit them referring the reader e.g. to \cite{FJMM}.

Hereafter we shall drop the suffix $0$ in the case $n=1$ and write e.g. $E(z)$.

\subsection{Coproduct and $R$ matrix}\label{sec:copro}

We shall use an automorphism $\theta$ of $\E_n$, 
which interchanges the vertical and horizontal 
subalgebras \cite{Mi}, \cite{Mi1}.
We have 
$\theta(E_{i,0})=E_{i,0}$, $\theta(F_{i,0})=F_{i,0}$, 
$\theta(K_{i})=K_{i}$ for $1\le i\le n-1$ 
and $\theta(C)= C^{\perp}$, $\theta(C^\perp)=C^{-1}$,
 $\theta(D)= D^{\perp}$, $\theta(D^\perp)=D^{-1}$.

Quite generally, 
we write $x^\perp=\theta(x)$ for an element $x\in \E_n$. 
We shall work mainly with  $E^\perp_i(z),F^\perp_i(z),K^{\pm,\perp}_i(z)$ 
which we call vertical currents. We refer to
 $E_i(z)$, $F_i(z)$, $K^\pm_i(z)$ as horizontal currents. 
\medskip

\noindent{\it Remark.}\quad 
We change the convention of \cite{FJMM2}, where 
$x^\perp$ was used to mean $\theta^{-1}(x)$.
\medskip

From now on, we drop $D$ from $\E_n$ and pass to the quotient
by the relation $C=1$. We denote the quotient algebra
by the same letter $\E_n$. 

Algebra $\E_n$ is equipped with a topological 
Hopf algebra structure
$(\Delta,\varepsilon,S)$
given by
\begin{align*}
&\Delta E^\perp_i(z)=E^\perp_i(z)\otimes 1+K^{-,\perp}_i(z)\otimes
E^\perp_i(C^\perp_1z) \,,
\\
&\Delta F^\perp_i(z)=F^\perp_i(C^\perp_2z)\otimes K^{+,\perp}_i(z)
+1\otimes F^\perp_i(z) \,,
\\
&\Delta K^{+,\perp}_i(z)
=K^{+,\perp}_i(C^\perp_2z)\otimes K^{+,\perp}_i(z)\,,
\\
&\Delta K^{-,\perp}_i(z)
=K^{-,\perp}_i(z)\otimes K^{-,\perp}_i(C^\perp_1z)\,,
\\
&\varepsilon(E^\perp_i(z))=0,\quad \varepsilon(F^\perp_i(z))=0,\quad 
\varepsilon(K^{\pm,\perp}_i(z))=1\,,\\
&S(E^\perp_i(z)) = -K^{-,\perp}_i((C^\perp)^{-1}z)^{-1} 
E^\perp_i((C^\perp)^{-1}z) ,\\
&S(F^\perp_i(z)) = -
F^\perp_i((C^\perp)^{-1}z) K^{+,\perp}_i((C^\perp)^{-1}z)^{-1},
\\
&
S(K^{\pm,\perp}_i(z))=K^{\pm,\perp}_i((C^\perp)^{-1}z)^{-1}\,.
\end{align*}
Here $C^\perp_1=C^\perp\otimes 1$ and $C^\perp_2=1\otimes C^\perp$. 
For $x=\Cp,D^\perp$, we set
$\Delta\, x=x\otimes x$, $\varepsilon(x)=1$, $S(x)=x^{-1}$.

Let $\B^\perp$ (resp. $\overline{\B}^\perp$) be the Hopf subalgebra of $\E_n$
generated by $E^\perp_{i,k},\ H^\perp_{i,-r}$ and 
$K^{\pm1}_i,(C^\perp)^{\pm1},(D^\perp)^{\pm1}$ 
(resp.  $F^\perp_{i,k},\ H^\perp_{i,r}$ and 
$K^{\pm1}_i,(C^\perp)^{\pm1},(D^\perp)^{\pm1}$)
where $i\in\Z/n\Z$, $k\in\Z$, $r>0$.

There exists a non-degenerate Hopf pairing, see \cite{N2},
$\langle~,~\rangle:\B^\perp\times \overline{\B}^\perp\to\C$ such that 
\begin{align*}
&\langle E_j^\perp(z),F^\perp_i(w)\rangle=\frac{1}{q-q^{-1}}\delta_{i,j}\delta(z/w)\,, \\
&\langle H_{j,-r}^{\perp},H_{i,s}^{\perp}\rangle
=\frac{1}{q-q^{-1}}a_{i,j}(r)\delta_{r,s}\,,\\
&\langle K^{-,\perp}_j,K^{+,\perp}_i\rangle=q^{-a_{i,j}}\,,
\quad 
\langle C^\perp,D^{\perp}\rangle=\langle D^\perp,C^{\perp}\rangle=q^{-1}\,.
\end{align*}
The quotient of the Drinfeld double of $\B^\perp$ by the relation 
$x\otimes 1-1\otimes x=0$ for $x=K_i^\perp,C^\perp,D^\perp$
is isomorphic to the Hopf algebra $\E_n$.

Denote by $\cN^\perp$ (resp. $\bar\cN^\perp$) 
the subalgebra generated by $E^\perp_{i,k}$ (resp. $F^\perp_{i,k}$), $i\in\Z/n\Z$, $k\in\Z$.
Set further $\bigl(\tilde{a}_{i,j}(r)\bigr)_{0\le i,j\le n-1}
=\bigl(a_{i,j}(r)\bigr)_{0\le i,j\le n-1}^{-1}$.  
Then the universal $R$ matrix of $\E_n$ has the form $\cR=\cR_1\cR_{0} q^{t_\infty}$,   
where 
\[
\cR_1=1+(q-q^{-1})\sum_{i=0}^{n-1}\sum_{k\in\Z}
E_{i,k}^\perp\otimes F_{i,-k}^\perp
+\cdots    
\]
is the canonical element of $\cN^\perp\otimes \bar\cN^\perp$, 
and
\begin{align*}
&\cR_{0}=\exp\Bigl((q-q^{-1})\sum_{0\le i,j\le n-1\atop r>0}
\tilde{a}_{i,j}(r)
H^{\perp}_{i,-r}\otimes H^\perp_{j,r}\Bigr)\,. 
\end{align*}
To describe $t_\infty$, introduce elements $H_{i,0}, c^\perp,d^\perp$ by 
setting $K^\perp_i=q^{H_{i,0}}$ ($1\le i\le n-1$), $C^\perp=q^{c^\perp}$, $D^\perp=q^{d^\perp}$.
Then 
\begin{align*}
&t_\infty=-c^\perp\otimes d^\perp-d^\perp\otimes c^\perp
+\sum_{i=1}^{n-1} \bar{\Lambda}_i\otimes H_{i,0}\,,
\end{align*}
where $\bar{\Lambda_i}=\sum_{j=1}^{n-1}\tilde{a}_{i,j}H_{j,0}$, 
$\bigl(\tilde{a}_{i,j}\bigr)_{1\le i,j\le n-1}
=\bigl(a_{i,j}\bigr)_{1\le i,j\le n-1}^{-1}$.

\medskip

\subsection{Fock module }\label{sub:Fock}
The most basic  $\E_n$ modules are the Fock modules $\F_\nu(u)$ ($\nu\in\Z/n\Z$).
In terms of the horizontal currents, it is characterized 
as the irreducible cyclic module generated by a vector $\ket{\emptyset}$,  
such that $E_i(z)\ket{\emptyset}=0$ for $i\in\Z/n\Z$ and
\[
K^{\pm}_i (z)\ket{\emptyset}=
\begin{cases}
\displaystyle{q\frac{1-q_2^{-1}u/z}{1-u/z}}\ket{\emptyset}& (i=\nu), \\
\ket{\emptyset}&(i\neq \nu).\\
\end{cases}
\]
The module $\F_\nu(u)$ has a basis labeled by partitions, which is an eigenbasis 
for $K^{\pm}_i (z)$ with explicit eigenvalues.
Moreover, 
the action of $E_i(z)$ and $F_i(z)$  is explicitly given by 
formulas similar to the Pieri rule, see \cite{FJMM2}.

Alternatively  $\F_\nu(u)$ has a realization by vertex operators 
\cite{Sa}, \cite{STU} in terms of vertical currents.
We recall the formulas below.

First assume $n\ge3$. 
Let $\bar\alpha_i$, $\bar\Lambda_i$ ($1\le i\le n-1$) be the simple roots and the fundamental
weights of $\mathfrak{sl}_n$,  
$\bar P=\bigoplus_{i=1}^{n-1}\Z\bar\Lambda_i$, $\bar Q=\bigoplus_{i=1}^{n-1}\Z\bar\alpha_i$, 
and  $(~,~)$ the standard symmetric bilinear form. 
We set $\bar\alpha_0=-\sum_{i=1}^{n-1}\bar\alpha_i$, $\bar\Lambda_0=0$.  
Introduce the twisted group algebra $\C\{\bar P\}$ generated by the symbols 
$e^{\pm \bar\alpha_2},\dots,e^{\pm \bar\alpha_{n-1}},e^{\pm\bar\Lambda_{n-1}}$
with the defining relations
\begin{align*}
 &e^{\bar\alpha_i} e^{\bar\alpha_j}=(-1)^{(\bar \alpha_i,\bar \alpha_j)}
e^{\bar\alpha_j}e^{\bar\alpha_i}\,,
\quad
 e^{\bar\alpha_i} e^{\bar\Lambda_{n-1}}=
(-1)^{\delta_{i,n-1}}e^{\bar\Lambda_{n-1}} e^{\bar\alpha_i} \,.
\end{align*}
For $\beta=\sum_{i=2}^{n-1}m_i\bar\alpha_i+m_n\bar\Lambda_{n-1}$ we set
\begin{align*}
 e^{\beta}=(e^{\bar\alpha_2})^{m_2}\cdots
(e^{\bar\alpha_{n-1}})^{m_{n-1}}(e^{\bar\Lambda_{n-1}})^{m_n}\,.
\end{align*}
Denote by $\C\{\bar Q\}$ the subalgebra generated by $e^{\bar\alpha_i}$ ($1\le i\le n-1$). 

Set $\cH=\C[H^{\perp}_{i,-m}\mid m>0,0\le i\le n-1]$. For each $\nu$ the Fock module 
$\F_\nu(u)$
is realized on $\cH\otimes\C\{\bar Q\}e^{\bar\Lambda_{\nu}}$. 
For $v\otimes e^\beta\in\cH\otimes\C\{\bar Q\}e^{\bar\Lambda_{\nu}}$, 
$\beta=\sum_{k=1}^{n-1}m_k\bar\alpha_k+\bar{\Lambda}_\nu$ 
and $0\le i\le n-1$, define
\begin{align*}
&z^{H_{i,0}}\bigl(v\otimes e^\beta\bigr) =z^{(\bar \alpha_i,\beta)}
d^{\frac{1}{2}\sum_{k=1}^{n-1}(\bar \alpha_i,m_k\bar\alpha_k) m_{ik}}
\bigl(v\otimes e^\beta\bigr) \,,
\\
&\partial_{\bar\alpha_i}
\bigl(v\otimes e^\beta\bigr) =
(\bar \alpha_i,\beta)
\bigl(v\otimes e^\beta\bigr) \,,
\\
&d^\perp(v\otimes e^\beta)=\bigl(\deg v+\frac{(\bar\beta,\bar\beta)}{2}
-\frac{(\bar\Lambda_\nu,\bar\Lambda_\nu)}{2}\bigr)(v\otimes e^\beta),
\end{align*}
where if $v=H^{\perp}_{i_1,-k_1}\cdots H^\perp_{i_N,-k_N}\in\cH$
then $\deg v=\sum_{j=1}^Nk_j$.

The action of the generators are given as follows.
\begin{align*}
&E_i^\perp(z)\mapsto c^*_i
\exp\Bigl(\sum_{r>0}\frac{q^{-r}}{[r]}H^\perp_{i,-r}z^r\Bigr) 
\exp\Bigl(-\sum_{r>0}\frac{1}{[r]}H^\perp_{i,r}z^{-r}\Bigr)
\otimes e^{\bar{\alpha}_i}z^{H_{i,0}+1} \,,
\\
&F_i^\perp(z)\mapsto c_i
\exp\Bigl(-\sum_{r>0}\frac{1}{[r]}H^\perp_{i,-r}z^r\Bigr) 
\exp\Bigl(\sum_{r>0}\frac{q^{r}}{[r]}H^\perp_{i,r}z^{-r}\Bigr)
\otimes e^{-\bar{\alpha}_i}z^{-H_{i,0}+1} \,,
\\
&K_i^{\pm,\perp}(z)\mapsto\exp\Bigl(\pm(q-q^{-1})
\sum_{r>0}H^\perp_{i,\pm r}z^{\mp r}\Bigr)
\otimes q^{\pm\partial_{\bar{\alpha}_i}}\,,
\\
&C^\perp \mapsto q\,,\quad 
D^\perp\mapsto q^{d^\perp}\,.
\end{align*}
The parameters $c_i$ are related to the spectral parameter $u$ by
\begin{align*}
c_0c_1\cdots c_{n-1} =(-1)^{n(n+1)/2+1}
q^{-1}d^{n/2+(n-2)\delta_{\nu,0}}u\,,\quad
c_i^*c_i=1\,.
\end{align*}

The same formulas apply for $n=1,2$; the only changes are that 
$\C\{\bar P\}$ is replaced by the ordinary group algebra for $n=2$, 
and that it is absent for $n=1$. The spectral parameter reads
$c_0c_1=q_1u$, $c_i^*c_i=1$ for $n=2$, and $c^*_0=\bigl((1-q_1)(1-q_3)u\bigr)^{-1}$, $c_0=q^{-1}u$ 
for $n=1$. 

We shall deal with tensor products $W$ of various Fock spaces.
We say that a vector $w\in W$ has principal degree $N$ and 
weight $\beta\in \bar P$ if 
$N=nN_0-(\overline{\Lambda}_1+\cdots+\overline{\Lambda}_{n-1},\beta)$, 
$q^{d^\perp}w=q^{N_0}w$, 
and 
$K_iw=q^{(\bar\alpha_i,\beta)}w$, $i=0,\ldots,n-1$,

\section{Integrals of motion}\label{sec:IM}

\subsection{Transfer matrix}
Let $\bar p, \bar p_1,\dots, \bar p_{n-1}$ 
be formal variables. We call these variables twisting variables.
The transfer matrices associated with Fock representations 
are the following weighted traces of the universal $R$ matrix:
\begin{align*}
&
T_\nu(u)=\Tr_{\F_\nu(u),1}
\Bigl(
\bigl(\bar p^{\,d^\perp}\prod_{i=1}^{n-1} {\bar p}_i^{\,-\bar\Lambda_i}\bigr)_1
\ \cR_{12}\Bigr)q^{d^\perp}
\,,
\\
&T^*_\nu(u)=\Tr_{\F_\nu(u),1}\Bigl(\bigl(
\bar p^{\, d^\perp}\prod_{i=1}^{n-1} \bar p_i^{\, -\bar\Lambda_i} \bigr)_1
\ \cR_{21}^{-1}\Bigr)q^{-d^\perp}\,.
\end{align*}
Here the trace is taken in the first tensor component, 
and $\cR_{21}=\sigma(\cR)$, $\sigma(a\otimes b)=b\otimes a$. 
Up to an overall power, $T_\nu(u)$ (resp. $T_\nu^*(u)$)
are formal power series in $u^{-1}$ (resp. $u$),
whose coefficients are formal power series in $\bar p$. 
Each coefficient of $u^{\pm1}$ is a well defined operator on 
appropriate modules, 
including, in particular, various tensor products of Fock modules. 
Thanks to the Yang-Baxter equation, these series commute with each other:
\[
[T_\mu(u),T_\nu(v)]=[T^*_\mu(u),T_\nu(v)]=[T^*_\mu(u),T^*_\nu(v)]=0
\quad (\forall \mu,\nu,u,v).
\]

Below we present explicit formulas for the
Taylor coefficients in $u^{\pm1}$ of these transfer matrices. 
More specifically, let 
\begin{align}\label{pp}
p=\bar pq^{-c^\perp}\,,\quad   
p^*=\bar pq^{c^\perp}\,,
\end{align}

and introduce the dressed currents
\begin{align}
&\hF_i^\perp(z)=F_i^\perp(z)\hK_i^{+,\perp}(z)^{-1}\,,\quad 
\hK_i^{+,\perp}(z)=
\prod_{\ell=0}^\infty\bK_i^{+,\perp}(p^{-\ell}z)\,,
\label{Fhat}
\\
&\hE_i^\perp(z)=\hK_i^{-,\perp}(z)^{-1}E_i^\perp(z)\,,\quad 
\hK_i^{-,\perp}(z)=\prod_{\ell=0}^\infty \bK_i^{-,\perp}
((p^{*})^\ell z)\,.
\label{Ehat}
\end{align}

We will show that 
the Taylor coefficients are integrals of products of the dressed currents, 
with kernel functions written in terms of infinite products 
\begin{align*}
&(z_1,\dots,z_l;p)_\infty=\prod_{j=1}^l\prod_{k\ge0}(1- p^kz_j)\,,
\quad
\Theta_{ p}(z)=(z,  pz^{-1},p;p)_\infty\,.
\end{align*} 
Thus, the Taylor coefficients of
$T_\nu(u)$ (resp. $T_\nu^*(u)$) in $u^{-1}$ (resp. $u$) are convergent series in $p$.

For the sake of concreteness, we shall 
restrict the discussion to the case of $\E_1$ and $\E_2$.
Extension to general $\E_n$
is straightforward.

\subsection{Shuffle algebra}

The calculation of the transfer matrix is essentially done in 
\cite{FT} using the language of the shuffle algebra. 
We first recall the result of \cite{FT} in the simplest case of $\E_1$. 

The shuffle algebra is a graded algebra
$Sh=\oplus_{k=0}^\infty Sh_{k}$ defined as follows.  
Each graded component
$Sh_{k}$ consists of symmetric rational functions in $k$ variables of 
the form 
\begin{align*}
G(x_1,\dots,x_k)=\frac{g(x_1,\dots,x_k)}{\prod_{1\le i<j\le k}(x_i-x_j)^2}, 
\quad g(x_1,\dots,x_k)\in \C[x_1^{\pm1},\dots,x_k^{\pm1}]^{\GS_k}, 
\end{align*}
satisfying the so-called wheel condition
\begin{align*}
g(x_1,\dots,x_k)=0
\quad \text{if $(x_1,x_2,x_3)=(x,q_1x,q_1q_2x)$ or $(x,q_2x,q_1q_2x)$}. 
\end{align*}
For $G_1\in Sh_{k}$ and $G_2\in Sh_{l}$ the multiplication $\ast$ is defined by 
\begin{align*}
(G_1\ast G_2)(x_1,\dots,x_{k+l})=
\Sym\Bigl[
G_1(x_1,\dots,x_k)G_2(x_{k+1},\dots,x_{k+l})
\prod_{\genfrac{}{}{0pt}{}{1\le i\le k}{ 1\le j\le l}}
\prod_{s=1}^3\frac{x_{k+j}-q_s x_i}{x_{k+j}-x_i}\Bigr]\,,
\end{align*}
where
\begin{align*}
&\Sym\ G(x_1,\dots,x_N) =\frac{1}{N!}
\sum_{\pi\in\GS_N} G(x_{\pi(1)},\dots,x_{\pi{(N)}})\,.
\end{align*}
Denote by $\N^\perp$ the subalgebra of $\E_1$ generated by
$\{E^\perp_i\mid i\in\Z\}$. 
It is known, see \cite{N1}, that $Sh$ is generated by the subspace $Sh_{1}$ of degree one, 
and that there is an isomorphism of algebras $\tau: \ \N^\perp\simeq Sh$ 
such that 
\begin{align*}
 \tau(E^\perp_i)= x^i\in  Sh_{1}, \qquad i\in\Z\,.
\end{align*}
One can introduce further an extended shuffle algebra $Sh^{\ge}$ by adjoining 
formal elements $H^\perp_{-r}$ ($r>0$), $q^{c^\perp},q^{d^\perp}$ to $Sh$,    
in such a way that the map $\tau$ extends to an isomorphism 
$\tau:\ \B^\perp\simeq Sh^{\ge}$. See \cite{N1} for the details. 

Consider a linear functional 
$\phi:\overline{\B}^\perp\to \C[[\bar p]]$ given by 
$\phi(Y)=\Tr_{\F(u)}\bigl(\bar p^{\,d^\perp}Y\bigr)$. 
Applying $\phi$ to the second component of the $R$ matrix, we obtain 
\begin{align}
X=(\id\otimes\phi)\cR= \Tr_{\F(u),2}\bigl((
\bar p^{\,d^\perp}
)_2\cR_{12}\bigr)
\quad\in 
\B^\perp[[\bar p,u]]\,.
\label{X} 
\end{align}
Since $\cR$ is the canonical element, $X$ is uniquely 
characterized by the property
$\langle X,Y\rangle=\phi(Y)$ for all $Y\in \overline{\B}^\perp$. 
By computing both sides for a spanning set of elements 
in $\overline{\B}^\perp$,  
the following result is obtained in \cite{FT}.

\begin{prop}\label{prop:FT}
The element $\tau(X)\in Sh^{\ge}[[\bar p,u]]$ corresponding to \eqref{X}
is given by
\begin{align}
&\sum_{N=0}^\infty(q^{-1}u)^N\frac{(q-q^{-1})^N}{N!}
\frac{1}{(p;p)_\infty} 
\prod_{1\leq i<j\leq N}\frac{(x_i-q_2x_j)(x_i-q_2^{-1}x_j)}
{(x_i-x_j)^2}
\label{shuffle-X}\\
&\times
\prod_{1\le i,j\le N}
\frac{(px_j/x_i,pq_2x_j/x_i;p)_\infty}
{(pq_1^{-1}x_j/x_i,pq_3^{-1}x_j/x_i;p)_\infty}
\cdot
\prod_{i=1}^N\tilde{K}^{-,\perp}(x_i)\cdot q^{-d^\perp}\,.\nn
\end{align}
where $p$ is given by \eqref{pp}
and 
$\tilde{K}^{-,\perp}(z)=\prod_{\ell=1}^\infty \bK^{-,\perp}
(p^{\ell}z)$.
In the above expression, the terms should be ordered so that
all elements $H^\perp_{-r}$ are placed  to the right of all 
$x_i$'s.
\end{prop}

\subsection{Transfer matrix for $\E_1$}

In order to write down the transfer matrix we need to rewrite 
Proposition \ref{prop:FT} as operators acting on modules. 
We say that an $\E_1$ module $V$ has property (R)
if an arbitrary matrix element of a product of currents
$E^\perp(z)E^\perp(w)$ converges to a rational function of the form
\[
\bra{v^*}E^\perp(z)E^\perp(w)\ket{v}=\frac{p_{v^*,v}(z,w)}
{(z-q_1w)(z-q_2w)(z-q_3w)}\,,    
\]
where  $\ket{v}\in V$, $\bra{v^*}\in V^*$, and 
 $p_{v^*,v}(z,w)$ is a Laurent polynomial. 
Tensor products of Fock modules have property (R). 
From the quadratic relation of currents, 
it follows that on such modules the currents satisfy the commutation relation
\begin{align}
E^\perp(z)E^\perp(w)
=E^\perp(w)E^\perp(z)\cdot \prod_{s=1}^3\frac{1-q_s^{-1}w/z}{1-q_sw/z}
\label{EE-rel}
\end{align}
in the sense of matrix elements.

Below we shall use the symbols for ordered products
\[
\prod_{1\le i\le N}^{\curvearrowright}A_i=A_1A_2\cdots A_N\,,
\quad
\prod_{1\le i\le N}^{\curvearrowleft}A_i=A_NA_{N-1}\cdots A_1\,.
\]

\begin{lem}\label{lem:StoE}
Let $G\in Sh_N$ be an element of the shuffle algebra. 
Then the corresponding element $\tau^{-1}(G)\in \B^\perp$ acts 
on modules with property (R) 
by the formula
\begin{align*}
\int\!\!\cdots\!\!\int
\prod_{1\le i\le N}^{\curvearrowright} E^\perp(x_i)\cdot 
G(x_1,\dots,x_N)\prod_{i<j}\prod_{s=1}^3\frac{x_j-x_i}{x_j-q_sx_i}
\prod_{i=1}^N\frac{dx_i}{2\pi\sqrt{ -1} x_i}\,
\end{align*}
in the sense of matrix elements.
Here the integral is taken over a symmetric cycle such that
$q_sx_j$ is inside and $q_s^{-1}x_j$ is outside of the contour for $x_i$
 for each $i\neq j$ and $s=1,2,3$. 
\end{lem}
\begin{proof}
It suffices to show that if 
\[
G(x_1,\dots,x_N)=x^{m_1}*\cdots* x^{m_N}
=
{\rm Sym}\left(x_1^{m_1}\cdots x_N^{m_N}\prod_{i<j}
\prod_{s=1}^3\frac{x_j-q_sx_i}{x_j-x_i}\right)    
\]
then the integral reduces to $\tau^{-1}(G)=E^\perp_{m_1}\cdots E^\perp_{m_N}$. 
Due to \eqref{EE-rel} the integrand becomes
\begin{align*}
&\frac{1}{N!}\sum_{\pi\in\GS_N}\prod_{1\le i\le N}^{\curvearrowright} E^\perp(x_i)\cdot     
x_{\pi(1)}^{m_1}\cdots x_{\pi(N)}^{m_N}
\prod_{i<j\atop \pi(i)>\pi(j)}
\prod_{s=1}^3\frac{x_i-q_sx_j}{x_i-q_s^{-1}x_j}
\\
&=\frac{1}{N!}\sum_{\pi\in\GS_N}
\prod_{1\le i\le N}^{\curvearrowright} E^\perp(x_{\pi(i)})\cdot     
x_{\pi(1)}^{m_1}\cdots x_{\pi(N)}^{m_N}
\,.
\end{align*}
Since the contour is symmetric, all integrals reduce to the case
$\pi=\id$. In that case the contour can be changed to circles satisfying
$|x_1|\gg\cdots\gg|x_N|$, because matrix elements of 
$\ds{\prod_{1\le i\le N}^{\curvearrowright} E^\perp(x_{i})}$ 
have poles only at $x_i=q_s x_j$ ($i<j$).
The assertion follows from this.
\end{proof}

Now we apply Lemma \ref{lem:StoE} to Proposition \ref{prop:FT}, noting  
that 
\[
T^*(u)=
\Tr_{\F(u),1}\Bigl((
\bar p^{\,d^\perp}
)_1\cR_{21}^{-1}\Bigr)    
=S\Bigl(\Tr_{\F(u),2}\Bigl((
\bar p^{\,d^\perp} 
)_2\cR_{12}\Bigr)\Bigr)\,.    
\]
By using 
\[
E^\perp(z) K^{-,\perp}(w)=\prod_{s=1}^3\frac{w-q_s^{-1}z}{w-q_sz}\cdot
K^{-,\perp}(w)E^\perp(z)\,,
\]
and  moving the product of $\tilde{K}^{-,\perp}(x_i)$'s in \eqref{shuffle-X},  
we can rewrite the integrand in terms of the dressed currents \eqref{Ehat}.  
In the resulting integrand  the poles $x_i=q_2x_j$ are canceled, so that  
only the poles $x_i=q_1x_j,q_3x_j$ need to be taken into account. 
Proceeding similarly also for $T(u)$,
we arrive at the following.
Here and after we set
$\Theta_{ p}(z_1,\ldots,z_k)=\prod_{j=1}^k\Theta_{ p}(z_j)$.

\begin{prop}\label{gl1 transfer}
As operators on modules with property (R) we have
\begin{align}
T(u)=&\Tr_{\F(u),1}\Bigl(
(\bar p^{\,d^\perp})_1\cR_{12}\Bigr) q^{d^\perp}
=\sum_{N=0}^\infty (q^{-1}u)^{-N} c_N I_N\,,
\label{Tmat-F}
\end{align}
where
\begin{align}
&I_N=
\int\!\!\cdots \!\!\int
\prod_{1\le i\le N}^{\curvearrowright} \hF^\perp(x_i)\cdot 
\prod_{i<j}
\frac{\Theta_{p}(x_j/x_i,q_2x_j/x_i)}
{\Theta_{p}(q_1^{-1}x_j/x_i,q_3^{-1}x_j/x_i)}
\cdot\prod_{j=1}^N\frac{dx_j}{2\pi \sqrt{-1} x_j}
\,,
\label{IN*}\\
&c_N=\frac{1}{N!}
\frac{1}{(p;p)_\infty} 
\left(
\frac{(p,q_3q_1;p)_{\infty}}{(q_1,q_3;p)_{\infty}}
\right)^N
\,.
\nn
\end{align}
The contour of integration is taken as
$|x_1|=\cdots=|x_N|=1$ when  $|q_1|,|q_3|<1$,
and by analytic continuation in the general case.  

Similarly we have 
\begin{align*}
T^*(u)=&\Tr_{\F(u),1}\Bigl((\bar p^{\,d^\perp})_1
\cR_{21}^{-1}\Bigr) q^{-d^\perp}
=\sum_{N=0}^\infty (q^{-1}u)^{N}c^*_N \cdot  I^*_N\,,
\end{align*}
where
\begin{align*}
&I^*_N=
\int\!\!\cdots \!\!\int
\prod_{1\le i\le N}^{\curvearrowleft} \hE^\perp(x_i)\cdot 
\prod_{i<j}\frac{\Theta_{p^*}(x_j/x_i,q_2x_j/x_i)}
{\Theta_{p^*}(q_1^{-1} x_j/x_i,q_3^{-1}x_j/x_i)}
\cdot\prod_{j=1}^N\frac{dx_j}{2\pi \sqrt{-1} x_j}
\,,
\\
&c^*_N=\frac{1}{N!}\frac{(q^{-1}-q)^N}{(p^*;p^*)_\infty}
\left(\frac{(p^*,q_2p^*;p^*)_{\infty}}
{(q_1^{-1}p^*,q_3^{-1}p^*;p^*)_{\infty}}\right)^N
\,.\nn 
\end{align*}
The contour of integration is taken as
$|x_1|=\cdots=|x_N|=1$ when  $|q_1|,|q_3|>1$,
and by analytic continuation in the general case.  \qed 
\end{prop}

\subsection{Transfer matrix for $\E_2$}

For the quantum toroidal $\gl_n$ algebra $\mc E_n$ with $n\ge2$, 
the trace functional has been calculated in \cite{FT}
in terms of an appropriate version of the shuffle algebra. 
The corresponding formulas for the transfer matrices
can be extracted from there in a similar manner.  
Since the argument is the same as in the previous case,
we omit the details and state only the final result for $\E_2$. 

We use the theta functions 
\begin{align}\label{theta-nu}
 \vartheta_\nu(z,p)=\sum_{n\in\Z+\nu/2}p^{n^2}z^{2n}\,,
\quad \nu=0,1.
\end{align}

\begin{prop}\label{gl2 transfer}
As operators on modules with property (R), the following formulas hold.
\begin{align}
&T_\nu(u)=\Tr_{\F_\nu(u),1}   
\Bigl((
\bar p^{\,d^\perp} \bar p_1^{\,-\bar \Lambda_1 })_1
\cR_{12}\Bigr) q^{d^\perp}
=\sum_{N=0}^\infty (q_1u)^{-N}c_N G_{\nu,N}\,,
\label{Tm*}
\\
&G_{\nu,N}=
\int\!\!\cdots \!\!\int
\prod_{1\le i\le N}^{\curvearrowright} \hF_0^\perp(x_{0,i})
\prod_{1\le i\le N}^{\curvearrowright}\hF_1^\perp(x_{1,i})
\cdot 
\frac{\prod_{t=0,1}\prod_{i<j}
\Theta_{p}(x_{t,j}/x_{t,i},q_2x_{t,j}/x_{t,i})}
{\prod_{i,j}
\Theta_{p}(q_1^{-1}x_{1,j}/x_{0,i},q_3^{-1}x_{1,j}/x_{0,i})
}
\label{G*nuN}\\
&\times \prod_{i=1}^N(x_{0,i}x_{1,i})^{N-2i+1}
\left(\prod_{i=1}^N\frac{x_{1,i}}{x_{0,i}}\right)^N
p^{\,-\nu/4}
\vartheta_\nu
\Bigl(
\bar p_1^{1/2}q^{-H_{1,0}/2}  
\prod_{i=1}^N\frac{x_{1,i}}{x_{0,i}},
p 
\Bigr)
\prod_{1\le j\le N\atop t=0,1}\frac{d x_{t,j}}{2\pi \sqrt{-1} x_{t,j}}
\,,
\nn\\
&c_N=\frac{1}{N!^2}
(q-q^{-1})^{2N}
(p;p)_{\infty}^{2N-2}(p,q_2^{-1}p;p)^{2N}_{\infty}
\,.\nn
\end{align}
The contour of integration is taken as $|x_{t,j}|=1$ when  $|q_1|,|q_3|<1$,
and by analytic continuation in the general case.

Similarly we have
\begin{align*}
&T^*_\nu(u)=Tr_{\F_\nu(u),1}\Bigl((
\bar p^{\,d^\perp}\bar p_1^{-\bar \Lambda_1}
)_1
\cR_{21}^{-1}\Bigr) q^{-d^\perp}
=\sum_{N=0}^\infty 
(q_1u)^{N}  
c^*_N\cdot G^*_{\nu,N}\,,
\nn\\
&G^*_{\nu,N}=
\int\!\!\cdots \!\!\int
\prod_{1\le i\le N}^{\curvearrowleft}\hE_1^\perp(x_{1,i})
\prod_{1\le i\le N}^{\curvearrowleft} \hE_0^\perp(x_{0,i})
\cdot 
\frac{\prod_{t=0,1}\prod_{i<j}
\Theta_{p^*}(x_{t,j}/x_{t,i},q_2x_{t,j}/x_{t,i})}
{\prod_{i,j}
\Theta_{p^*}(q_1^{-1}x_{1,j}/x_{0,i},q_3^{-1}x_{1,j}/x_{0,i})
}
\\
&\times \prod_{i=1}^N(x_{0,i}x_{1,i})^{N-2i+1}
\left(\prod_{i=1}^N\frac{x_{1,i}}{x_{0,i}}\right)^N
(p^{*})^{-\nu/4}
\vartheta_\nu\Bigl(
\bar p^{1/2}_1q^{H_{1,0}/2}
\prod_{i=1}^N\frac{x_{1,i}}{x_{0,i}},p^*\Bigr)
\prod_{1\le j\le N\atop t=0,1}\frac{d x_{t,j}}{2\pi \sqrt{-1} x_{t,j}}
\,,
\nn\\
&c^*_N=\frac{1}{N!^2}
(q-q^{-1})^{2N}(p^*;p^*)_{\infty}^{2N-2}
(p^*,q_2^{-1}p^*;p^*)^{2N}_{\infty}
\,.\nn
\end{align*}

The contour of integration is taken as $|x_{t,j}|=1$ 
when  $|q_1|,|q_3|>1$,
and by analytic continuation in the general case.  
\qed
\end{prop}
\medskip

\noindent{\it Example.} The first non-trivial coefficient reads:
\begin{align*}
p^{\,\nu/4}
G_{\nu,1}
=\int\!\!\int 
\hF^\perp_0(z) \hF^\perp_1(w)
\frac{w}{z}
\frac{
\vartheta_\nu\bigl(\bar p^{1/2}_1q^{-H_{1,0}/2}w/z,p\bigr)}
{\Theta_{p}\bigl(q_1^{-1}w/z,q_3^{-1}w/z\bigr)}
\frac{dz}{2\pi \sqrt{-1} z}\frac{dw}{2\pi \sqrt{-1} w}\,.
\end{align*}

\section{Comparison with \cite{FKSW}}\label{sec:compare}
In this section, 
we compare the results of the previous section
with the integrals of motion of \cite{FKSW},   
see Corollaries \ref{gl1 compare}, \ref{gl2 compare}. 

\subsection{Integrals of motion of \cite{FKSW}}\label{sec FKSW}

First we review the results of \cite{FKSW}. 

We fix parameters $r,s$ and $x$ with $0<x<1$. 
Let $\beta^1_m,\beta^2_m$ ($m\in\Z\backslash\{0\}$) be oscillators satisfying
\begin{align*}
&[\beta^i_m,\beta^i_{n}]
= \delta_{m+n,0}\
m\frac{[(r-1)m]}{[rm]}\frac{[(s-1)m]}{[sm]}\,,
\\
&[\beta^i_m,\beta^j_{n}]=
-\delta_{m+n,0}\
m\frac{[(r-1)m]}{[rm]}\frac{[m]}{[sm]}x^{sm\,\sgn(i-j)}\quad   (i\neq j).
\end{align*}
Let further $P,Q$ be zero mode operators such that $[P,\sqrt{-1}Q]=2$. 
Consider the bosonic Fock spaces $\F_{l,k}$ ($l,k\in\Z$)
generated by $\beta_{-m}^i$ ($i=1,2$, $m>0$) over the vacuum 
state $\ket{l,k}$ such that
\begin{align*}
&\ket{l,k}=
e^{\bigl(l\sqrt{\frac{r}{r-1}}-k\sqrt{\frac{r-1}{r}}\bigr)
\frac{\sqrt{-1}}{2}Q}\ket{0,0}\,,
\\
&P\ket{l,k}=\Bigl(l\sqrt{\frac{r}{r-1}}-k\sqrt{\frac{r-1}{r}}\Bigr)\ket{l,k}\,.
\end{align*}
We set $\hat{\pi}=\sqrt{r(r-1)}P$. 

The current of the deformed Virasoro algebra (DVA) acts on each $\F_{l,k}$
by  
\begin{align}
\ssT_1(z)=
x^{-\hat{\pi}}:\exp\Bigl(\sum_{m\neq 0}
\tilde{\beta}^1_mz^{-m}  
\Bigr):
+x^{\hat{\pi}}:\exp\Bigl(\sum_{m\neq 0}
\tilde{\beta}^2_mz^{-m} 
\Bigr):\,,
\label{DVir}
\end{align}
where $\tilde{\beta}^i_m=(x^{rm}-x^{-rm})\beta^i_m/m$. 

The screening currents $\ssF_i(z)$ act on $\F=\oplus_{l,k}\F_{l,k}$. 
They are given by

\begin{align*}
&\ssF_0(z)=e^{-\sqrt{-1}\sqrt{\frac{r-1}{r}}Q}z^{-\frac{\hat{\pi}}{r}+\frac{r-1}r}
:\exp\Bigl(\sum_{m\neq 0}\frac{1}{m}(-x^{sm}\beta^1_m+x^{-sm}\beta^2_m)\Bigr):\,,
\\
&\ssF_1(z)=e^{\sqrt{-1}\sqrt{\frac{r-1}{r}}Q}z^{\frac{\hat{\pi}}{r}+\frac{r-1}r}
:\exp\Bigl(\sum_{m\neq 0}\frac{1}{m}(\beta^1_m-\beta^2_m)\Bigr):\,.
\end{align*}

In \cite{FKSW}, two families of operators $\{\ssI_N\}_{N\ge1}$, 
 $\{\ssG_N\}_{N\ge1}$ are introduced. Both operators act on each $\F_{l,k}$. 
They are so designed that in the conformal limit
$x\to 1$ they tend to the local and non-local integrals of motion
of \cite{BLZ1}, respectively. 
For that reason,  $\{\ssI_N\}_{N\ge1}$ are called ``local'' integrals of motion  and   
 $\{\ssG_N\}_{N\ge1}$ ``non-local'' integrals of motion in \cite{FKSW}, 
even though both are non-local expressions as we shall see below. To distinguish from \cite{BLZ1} we call $\{\ssI_N\}_{N\ge1}$, 
 $\{\ssG_N\}_{N\ge1}$ elliptic local integrals of motion and elliptic non-local integrals of motion respectively.

Assume $0<r<1$, $0<s<2$. We use
\[
[u]_t=x^{\frac{u^2}{t}-u}\frac{\Theta_{x^{2t}}(x^{2u})}
{(x^{2t};x^{2t})_\infty}\,\quad (t=r,s). 
\]
The elliptic local integrals of motion are defined in terms of the DVA currents 
\begin{align}
\ssI_N&
=\int\!\!\cdots\!\!\int
\prod_{j=1}^N\frac{dz_j}{2\pi \sqrt{-1} z_j}
\prod_{1\le i\le N}^{\curvearrowright} \ssT_1(z_j)
\prod_{1\le i<j\le N}
\frac{[u_j-u_i]_s[u_j-u_i+r]_s}{[u_j-u_i+1]_s[u_j-u_i+r-1]_s}\,.
\label{ssI}
\end{align}
Here $z_j=x^{2u_j}$. 
The contours encircle the origin in such a way that
the poles 
$z_j=x^{-2+2sl}z_k$, $x^{-2(r-1)+2sl}z_k$ ($l=0,1,2,\dots$) are inside and
$z_j=x^{2-2sl}z_k$, $x^{2(r-1)-2sl}z_k$ ($l=0,1,2,\dots$) are outside, 
for all pairs $j<k$. 

The elliptic non-local integrals of motion are defined using the screening currents. 
They also depend on the choice of a function $\vartheta(u)$, which  
is an entire function satisfying
\begin{align*}
\vartheta(u+r)=\vartheta(u)\,,\quad 
\vartheta(u+r\tau)=e^{-2\pi \sqrt{-1} \tau+2\pi \sqrt{-1}\frac{2u-\hat{\pi}}{r}}\vartheta(u)\,,
\end{align*}
where $\tau$ is related to $x$ by $x^{2r}=e^{-2\pi \sqrt{-1}/\tau}$. 
The linear space of such functions is two dimensional, and is spanned by
\[
\vartheta_\nu(u)=
x^{\frac{2u^2}{r}-\frac{2\hat{\pi}u}{r}}
\vartheta_\nu(x^{2u-\hat{\pi}},x^{2r})
\]
where $\vartheta_\nu(z,p)$ is given in \eqref{theta-nu}.
Exhibiting the dependence on $\vartheta_\nu(u)$, we have
\begin{align}
\ssG_{\nu,N}
&=\int\!\!\cdots\!\!\int
\prod_{j=1}^N\frac{dz_j}{2\pi \sqrt{-1} z_j}\prod_{j=1}^N\frac{dw_j}{2\pi \sqrt{-1} w_j}
\prod_{1\le i\le N}^{\curvearrowright} \ssF_1(z_i)
\prod_{1\le i\le N}^{\curvearrowright} \ssF_0(w_i)
\label{ssG}\\
&\times
\frac{\prod_{1\le i<j\le N}[u_j-u_i]_r[u_j-u_i-1]_r[v_j-v_i]_r[v_j-v_i-1]_r}
{\prod_{1\le i,j\le N}[v_j-u_i-s/2]_r[v_j-u_i+s/2-1]_r}\ 
\vartheta_\nu\bigl(\sum_{j=1}^N(u_j-v_j)\bigr)\,,
\nn
\end{align}
where $z_j=x^{2u_j}$,  $w_j=x^{2v_j}$. 
The contours are the unit circle $|z|=1$.

The following commutativity is established in 
\cite{FKSW}, \cite{KS}
by a direct 
calculation. 
\begin{prop}
The operators $\ssI_m$, $\ssG_{\nu,m}$ mutually commute:
\begin{align*}
[\ssI_k,\ssI_m]=[\ssG_{\mu,k},\ssG_{\nu,m}]=
[\ssI_k,\ssG_{\nu,m}]=0\quad (\forall k,m>0,\ \mu,\nu).
\end{align*}
\end{prop}

\subsection{Elliptic local integrals of motion and {$\E_1$}}
Consider the tensor product of $\E_1$ Fock modules $\F(u_1)\otimes \F(u_2)$.
The action of the dressed
current  
$\hF^\perp(z)$ 
reads
\begin{align}
\Delta \hF^\perp(z)
&
=q^{-1}u_1:\exp\Bigl(
\sum_{m\neq 0}B^1_{m}z^{-m}\Bigr) :
+q^{-1}u_2:\exp\Bigl(
\sum_{m\neq 0}B^2_{m}z^{-m}\Bigr):\,,
\label{Delta-hE}
\end{align}
where 
$p=\bar pq^{-2}$ 
and
\begin{align*}
&B^1_{-m}=-\frac{q^{m}}{[m]}H^\perp_{-m}\otimes 1\,,\quad 
B^1_m=
\frac{1-p^{m}q^{2m}}{1-p^{m}}
\frac{q^{-2m}}{[m]}H^\perp_m\otimes 1
-\frac{(q-q^{-1})p^{m}}{1-p^{m}}1\otimes H^\perp_m
\,,\\
&B^2_{-m}=-\frac{1}{[m]}1\otimes H^\perp_{-m}\,,\quad 
B^2_m=
-\frac{(q-q^{-1})q^{-m}}{1-p^{m}} 
H^\perp_m\otimes1+\frac{1-p^{m}q^{2m}}{1-p^{m}}
\frac{q^{-m}}{[m]}1\otimes H^\perp_m\,,
\end{align*}
for $m>0$.
They satisfy the relations
\begin{align*}
&[B^i_m,B^i_{-m}]=
-\frac{1}{m}\frac{1-p^{m}q^{2m}}{1-p^{m}}(1-q_1^m)(1-q_3^m)\,,
\\
&[B^i_m,B^j_{-m}]=-\frac{1}{m}
\frac{1}{1-p^{m}}
\prod_{s=1}^3(1-q_s^m)
\times
\begin{cases}
p^{m}
& (i=1,j=2),\\
1 & (i=2,j=1).\\
\end{cases}
\end{align*}
On the other hand, the oscillators $\tilde{\beta}^i_m$ entering \eqref{DVir} 
satisfy 
\begin{align*}
&\Bigl[\tilde{\beta}^i_m,\tilde{\beta}^i_{-m}\Bigr] 
=-\frac{1}{m}\frac{1-x^{(2s-2)m}}{1-x^{2sm}}(1-x^{-2(r-1)m})(1-x^{2rm})\,,
\\
&\Bigl[\tilde{\beta}^i_m, \tilde{\beta}^j_{-m} 
\Bigr] 
=-\frac{1}{m}\frac{1-x^{-2m}}{1-x^{2sm}}(1-x^{-2(r-1)m})(1-x^{2rm})
\times
\begin{cases}
 1 & (i=1,j=2),\\
x^{2sm} & (i=2,j=1).\\
\end{cases}
\end{align*}

Fix $l,k$, and identify the vector space $\mc F_{l,k}$ with $\mc E_1$ module $\mc F(u_1)\otimes \mc F(u_2)$
by setting
$B^i_m=\tilde{\beta}^{3-i}_m$ and
\begin{align}\label{q1q2q3}
q_1=x^{2(1-r)}\,,\quad q_2=x^{-2}\,,\quad q_3=x^{2r}\,,\quad 
p=x^{2s}
\,,\quad 
\frac{u_2}{u_1}=x^{-2\hat{\pi}}\,.
\end{align}
Note that $\hat \pi$ acts in $\mc F_{l,k}$ as a scalar and, following \cite{FKSW}, we denote this scalar by the same letter $\hat\pi$.

\begin{cor}\label{gl1 compare}
With identification of parameters \eqref{q1q2q3}, 
the local integrals of motion $\{\ssI_N\}$ 
defined in \cite{FKSW}, see \eqref{ssI}, 
and coefficients $\{I_N\}$ of $\E_1$ transfer matrix, 
see \eqref{IN*}, 
generate the same commutative algebra. 
\end{cor}
\begin{proof}
Under our identification up to a scalar multiple 
the formulas for the currents \eqref{DVir} and \eqref{Delta-hE} match.

The kernel functions in the integrands
of \eqref{IN*} and \eqref{ssI} are different, 
but both of them are generators of the same commutative subalgebra
of a Feigin-Odesskii algebra \cite{FO}. 
For more details see \cite{FHHSY}, Section 4.
\end{proof}
\medskip

Following \cite{FKSW} let us extract 
from $\Delta \hat{F}^\perp(z)$ 
the contribution of the diagonal Heisenberg algebra by setting
\begin{align*}
&\Delta \hF^\perp(z)=\ssT^{DV}(z)\Delta(Z(z))\,, 
\\
&Z(z)=\exp\bigl(-\sum_{m>0}\frac{H^\perp_{-m}}{[2m]}z^m\bigr)
\exp\bigl(\sum_{m>0}
\frac{q^{-2m}-p^{m}q^{2m}}{1-p^{m}}
\frac{H^\perp_{m}}{[2m]}z^m\bigr)\,.
\end{align*}
The reduced current is independent of 
$p$,   
\begin{align}
&\ssT^{DV}(z)=x^{-\hat{\pi}}
:\exp\bigl(\sum_{m\neq 0}\frac{(qz)^{-m}}{[2m]}\bar{H}^\perp_m\bigr):
+x^{\hat{\pi}}
:\exp\bigl(-\sum_{m\neq 0}\frac{(q^{-1}z)^{-m}}{[2m]}\bar{H}^\perp_m\bigr):\,,
\label{TDV} 
\end{align}
and is written in terms of a single set of oscillators
\[
\bar{H}^\perp_{-m}=q^m H^\perp_{-m}\otimes 1-1\otimes H^\perp_{-m}\,,
\quad 
\bar{H}^\perp_m=H^\perp_m\otimes 1-q^{-m}1\otimes H^\perp_m
\]
which commute with the $\Delta{H}^\perp_m$'s.
It is in the form  \eqref{TDV} that 
the DVA current was originally introduced in  \cite{SKAO}.

\subsection{Elliptic
non-local integrals of motion and $\E_2$}
Choose $\mu\in\{0,1\}$ and 
consider the Fock representation $\F_{\mu}(v)$ of $\E_2$. 

Let
\begin{align}
q_1=x^{2-s}\,,\quad q_2=x^{-2}\,,\quad q_3=x^{s}\,,
\label{q1q2q3vee}
\end{align}
and match 
\begin{align*}
&\frac{1}{[m]}H^\perp_{1,-m}
=\frac{1}{m}\bigl(-x^{-sm}\beta^1_{-m}+x^{sm}\beta^2_{-m}\bigr)\,,
\quad
\frac{1}{[m]}\frac{[(r-1)m]}{[rm]}
H^\perp_{1,m}
=\frac{1}{m}\bigl(-x^{sm}\beta^1_{m}+x^{-sm}\beta^2_{m}\bigr)\,,
\\
&\frac{1}{[m]}H^\perp_{0,-m}
=\frac{1}{m}\bigl(\beta^1_{-m}-\beta^2_{-m}\bigr)\,,
\quad
\frac{1}{[m]}\frac{[(r-1)m]}{[rm]}
H^\perp_{0,m}
=\frac{1}{m}\bigl(\beta^1_{m}-\beta^2_{m}\bigr)\,.
\end{align*}

We choose any $k,l\in\Z$.
Then $\ssG_{\nu,N}$ preserve $\mc F_{k,l}$ and $G_{\nu,N}$ preserve the space $\mc F_{\mu,k}(u):={\mc H}\otimes e^{(k+\mu/2)\bar\al_1}$. The above identification of oscillators identifies $\mc F_{k,l}$ with $\mc F_{\mu,k}(u)$ as vector spaces.

We also set:
\begin{align}\label{p1-pi}
\quad 
p=x^{2r} \,, \qquad \bar p_1 q^{-H_{1,0}}=x^{-2\hat{\pi}}.
\end{align}
Then we obtain the following statement.
\begin{cor}\label{gl2 compare} 
Under identification of parameters \eqref{q1q2q3vee} and \eqref{p1-pi}, 
the non-local integrals of motion of $\ssG_{\nu,N}$  defined in \cite{FKSW} and acting in $\mc F_{k,l}$, 
see \eqref{ssG}, and the coefficients $G_{\nu,N}$ of the $\E_2$ transfer matrix,  acting in $\mc F_{\mu,k}(u)$, see \eqref{G*nuN}, coincide.
\end{cor} 
\begin{proof}
Note that currents $F^\perp_i(z)$ and $\ssF_{1-i}(z)$ 
have the same oscillator part. The corollary follows. 
\end{proof}

In particular, it follows that the non-local integrals $\ssG_{\nu,N}$'s commute because 
the transfer matrices commute. 
This fact was checked in \cite{FKSW} by a direct computation.

\section{Bethe Ansatz}\label{sec:BA}

\subsection{Bethe ansatz for $\E_1$}  

Bethe ansatz for the $\E_1$ 
transfer matrices was studied in \cite{FJMM1}, \cite{FJMM2}.  
We consider the transfer matrix $T(u)$ given in \eqref{Tmat-F} acting
on a tensor product of Fock spaces 
\[
W=\F(v_1)\otimes\cdots\otimes \F(v_M)\,.    
\]

We denote $\mc P$ the set of all partitions. 
Each partition $\la=(\la_1,\ldots,\la_\ell)$, 
$\la_1\ge\cdots\ge\la_\ell>0$, 
is represented also 
as a set of points on the plane  
$Y_\lambda=\{(i,j)\in\Z^2\mid 1\le i\le \ell\,, 1\le j\le \la_i\}$.
If $(i,j)\in Y_\lambda$ we say that the node $(i,j)$ belongs to $\la$
and write $(i,j)\in\la$.

The following result is obtained in \cite{FJMM2}.
\begin{prop}\label{BA gl1 prop}
For each eigenvector $w\in W$ of $T(u)$ 
of principal degree $N$, there exists a polynomial 
$Q(u)=\prod_{i=1}^N(1-t_i/u)$ in $u^{-1}$ such that the following hold. 
Set 
\begin{align*}
&\ab(u)=
p
\prod_{j=1}^M\frac{1-v_j/u}{1-q_2^{-1}v_j/u}
\prod_{s=1}^3\frac{Q(q_su)}{Q(q_s^{-1}u)}\,,\quad
p=\bar pq^{-M}\,.  
\end{align*}
Then the zeroes $\{t_i\}$ of $Q(u)$ satisfy the Bethe ansatz equation
\begin{align}\label{BAE gl1}
p\,  
\prod_{j=1}^M\frac{t_i-v_j}{t_i-q_2^{-1}v_j}
\prod_{j=1}^N
\frac{(q_1t_i-t_j)(q_2t_i-t_j)(q_3t_i-t_j)}
{(q_1^{-1}t_i-t_j)(q_2^{-1}t_i-t_j)(q_3^{-1}t_i-t_j)}
=-1,\qquad i=1,\ldots,N\,.
\end{align}
Denoting the corresponding eigenvalue of $T(u)$ by the same letter, we have
\begin{align}
&T(u)=
\varphi(u)\frac{Q(q_2^{-1}u)}{Q(u)}\sum_{\la\in \cP}
\prod_{\square\in\lambda}\mathfrak{a}(q^{-\square} u)\,,
\label{T-eigv}\\
&\varphi(u)=
\prod_{j=1}^M
\exp\Bigl(\sum_{r>0}
\frac{1}{r}\frac{1-q_2^{-r}}{(1-q_1^{r})(1-q_3^{r})}\bigl(\frac{v_j}{u}\bigr)^r
\Bigr)\,.\nn
\end{align}
Here $\lambda$ runs over all partitions, 
$\square=(a,b)$ runs over the nodes of 
$\lambda$ and $q^{\square}=q_3^{a-1}q_1^{b-1}$.
\end{prop}
By taking coefficients of $u^{-\ell}$ in \eqref{T-eigv}
the eigenvalues of the elliptic local integrals of motion can be expressed
in terms of symmetric polynomials of the Bethe roots $\{t_j\}$.  
Namely, set 
\begin{align*}
&\bw_r=\frac{\kappa_r}{r}\Bigl(\sum_{j=1}^N t^r_j
-\frac{q_3^rq_1^r}{(1-q_3^r)(1-q_1^r)} 
\sum_{k=1}^M v^r_k\Bigr)\,,
\quad \kappa_r=\prod_{s=1}^3(1-q_s^r).
\end{align*} 
\begin{cor}\label{eigvI}
Notation being as above, the elliptic local integral of motion \eqref{IN*} has the eigenvalue
\begin{align*}
&q^\ell c_\ell I_\ell=
\sum_{\alpha}C_\alpha\prod_{r\ge1}\frac{\bw_r^{\alpha_r}}{\alpha_r!}\,,
\quad
\end{align*} 
where the sum extends over all 
$\alpha=(1^{\alpha_1}2^{\alpha_2}\cdots)$ 
such that $\alpha_r\in\Z_{\ge0}$, $\sum_{r\ge1}r\alpha_r=\ell$.
The coefficients $C_\alpha$ are given by
\begin{align*}
&C_\alpha=\sum_{\lambda\in\cP}
p^{|\lambda|} 
\prod_{r\ge1}
\bz_r(\lambda)^{\alpha_r}\,,
\quad 
\bz_r(\lambda)
=\frac{1}{(1-q_3^r)(1-q_1^r)}-\sum_{(a,b)\in\lambda}q_3^{(a-1)r}q_1^{(b-1)r}\,.
\end{align*} 
\qed
\end{cor}
It turns out that the leading coefficient factorizes
(see the identity in Proposition \ref{identity}):
\begin{align}\label{identity2}
C_{\ell}
=\sum_{\lambda\in\cP}
p^{|\lambda|}\bz_\ell(\la)
=\frac{(p\,q_3^\ell q_1^\ell;p)_\infty}
{(q_3^\ell,q_1^\ell;p)_{\infty}}\,.
\end{align}
As a result, we obtain the following technical statement.

Let $R=\max(1,|q_1|,|q_3|)$.
\begin{prop}\label{convergent prop}
For all $\ell$, 
the coefficients of $u^{-\ell}$ of series $T(u)$ and $Q(u)$ 
are convergent series of $p$ with radius of convergence at least $R^{-2\ell}$.

For generic $p, v_1,\dots,v_M$, 
the operator $I_1$ acting in $W$ is diagonalizable 
and has a simple spectrum.
\end{prop}
\begin{proof}
The coefficients of $u^{-\ell}$ of series $T(u)$ are given by explicit formulas 
involving vertex operators 
(see \eqref{IN*} for $M=2$ and in \cite{FKSW2} 
 for general $M$), and  are convergent in
$|p|<R^{-2}$.  
By Corollary \ref{eigvI}, the corresponding eigenvalues
\[
\sum_{\alpha}C_\alpha\prod_{r\ge1}\frac{\bw_r^{\alpha_r}}{\alpha_r!}
=C_\ell
\bw_\ell+\text{(terms with $\bw_j$, $j<\ell$)}
\]
are holomorphic in $|p|<R^{-2}$.  
It is easy to show 
the following estimate 
\begin{align*}
\prod_{r\ge1}|\bz_r(\la)|^{\alpha_r} \le K (|\la|R^{|\la|})^{2\ell}
\qquad (\sum_{r\ge1}r\alpha_r=\ell)
\end{align*}
with some constant $K>0$. 
Hence the coefficients $C_\alpha$ are convergent in 
$|p|<1/R^{2\ell}$.   
Moreover, due to \eqref{identity2} 
the leading coefficient $C_\ell$ 
does not vanish there.
Therefore, by induction on $\ell$, 
all $\bw_\ell$ are convergent power series of 
$p$.   
This implies that each coefficient of $u^{-\ell}$ in 
the eigenvalue $Q(u)$ is a convergent power series in 
$p$. 

For $p=0$ the operator $I_1$ becomes $H_1$ which is diagonalizable and has simple spectrum, if $v_i$ are generic.
Therefore, the same is true generically.
\end{proof}

The spectrum of the other transfer matrix $T^*(u)$ can be described by Bethe ansatz method in a similar way.

\subsection{Bethe ansatz for $\E_2$}
In \cite{FJMM2}, the Bethe ansatz \eqref{BAE gl1} was derived by 
constructing auxiliary operators $Q(u)$ and $\mathcal{T}(u)$, which 
mutually commute and satisfy a two-term Baxter's relation 
$\mathcal{T}(u)=a(u)\prod_{s=1}^3Q(q_s^{-1}u)+pd(u)\prod_{s=1}^3Q(q_su)$. 
Similarly to the transfer matrix $T(u)$, 
these operators are defined as traces of $\cR$, but now 
over representations of a Borel subalgebra $\mathcal{B}^\perp$ of $\E_1$. 
In particular, $Q(u)$ corresponds to the so-called 
fundamental module $M^+(u)$ characterized by highest weight $K(z)v=(1-u/z)v$. 
The corresponding eigenvalues of the transfer matrix $T(u)$ 
are then obtained by writing the $q$-character of $\F(u)$ and 
substituting each term by appropriate ratios of $Q(u)$. 

Naturally we expect that, for all $\E_n$ algebras,  
the same logic should persist for the description of 
transfer matrix eigenvalues.  
The resulting formulas are easy to guess. 
Let us formulate below the Bethe ansatz in the case of $\E_2$. 

We consider $T_\nu(u)$ in \eqref{Tm*}     
on a tensor product of Fock spaces
\[
W=\F_0(v_{0,1})\otimes\cdots\otimes \F_0(v_{0,M_0})\otimes    
\F_1(v_{1,1})\otimes\cdots\otimes \F_1(v_{1,M_1})\,.
\]

\begin{conj}\label{gl2 BAE conj}
For each eigenvector $w\in W$ of $T_\nu(u)$ 
of principal degree $N_0+N_1$ and weight $(N_0-N_1)\bar\alpha_1$, 
there exist polynomials
$Q_0(u)=\prod_{i=1}^{N_0}(1-s_i/u)$, 
$Q_1(u)=\prod_{i=1}^{N_1}(1-t_i/u)$ with the following properties. 
Set 
\begin{align*}
&\ab_0(u)=
p_0
\prod_{l=1}^{M_0}\frac{1-v_{0,l}/u}{1-q_2^{-1}v_{0,l}/u}\cdot
\frac{Q_0(q_2 u)}{Q_0(q_2^{-1} u)}
\frac{Q_1(q_3 u)}{Q_1(q_3^{-1} u)}\frac{Q_1(q_1 u)}{Q_1(q_1^{-1} u)}\,,
\\
&\ab_1(u)=
p_1
\prod_{l=1}^{M_1}\frac{1-v_{1,l}/u}{1-q_2^{-1}v_{1,l}/u}\cdot
\frac{Q_1(q_2 u)}{Q_1(q_2^{-1} u)}
\frac{Q_0(q_3 u)}{Q_0(q_3^{-1} u)}\frac{Q_0(q_1 u)}{Q_0(q_1^{-1} u)}\,,
\\
&
{p_0=\bar p \bar p_1^{\,-1}q^{-M_0}\,,\quad 
p_1=\bar p_1 q^{-M_1}\,.
}
\end{align*}
Then the roots $\{s_i\}$, $\{t_j\}$ satisfy the Bethe ansatz equations
\begin{align}
&p_0
\prod_{l=1}^{M_0}\frac{s_i-v_{0,l}}{s_i-q_2^{-1}v_{0,l}}\cdot
\prod_{j=1}^{N_0}\frac{q_2s_i-s_j}{q_2^{-1}s_i-s_j}
\prod_{k=1}^{N_1}\frac{(q_1s_i-t_k)(q_3s_i-t_k)}
{(q_1^{-1}s_i-t_k)(q_3^{-1}s_i-t_k)}
=-1\,,\quad (1\le i\le N_0 )\,,
\label{BAE gl2}\\
&p_1
\prod_{l=1}^{M_1}\frac{t_i-v_{1,l}}{t_i-q_2^{-1}v_{1,l}}\cdot
\prod_{j=1}^{N_1}\frac{q_2t_i-t_j}{q_2^{-1}t_i-t_j}
\prod_{k=1}^{N_0}\frac{(q_1t_i-s_k)(q_3t_i-s_k)}
{(q_1^{-1}t_i-s_k)(q_3^{-1}t_i-s_k)}
=-1\,,\quad (1\le i\le N_1)\,.
\label{BAE gl2-2}
\end{align}
The corresponding eigenvalue of $T_\nu(u)$ is given by 
\begin{align*}
T_\nu(u)=
\varphi_{\nu}(u)\frac{Q_{\nu}(q_2^{-1}u)}{Q_{\nu}(u)}
\sum_{\la\in \cP}
\prod_{\square\in\lambda}\mathfrak{a}_{c_\nu(\square)}(q^{-\square}u)\,.
\end{align*}
Here $c_\nu(\square)\equiv a-b+\nu$ ($\bmod 2$) stands for the color of 
a node $\square=(a,b)\in\la$, and we set 
\begin{align*}
&\varphi_\nu(u) 
=\prod_{l=1}^{M_\nu}\Phi_{\nu,\nu}(v_{\nu,l}/u)
\prod_{k=1}^{M_{1-\nu}}\Phi_{\nu,1-\nu}(v_{1-\nu,k}/u)\,,
\\
&\Phi_{\nu,\nu}(v)=
\exp\Bigl(\sum_{r>0}\frac{1}{r}
\frac{(1-q_2^{-r})(1+q_2^{-r})}
{(1-q_1^{2r})(1-q_3^{2r})}
v^r\Bigr)\,,
\\
&\Phi_{\nu,1-\nu}(v)=
\exp\Bigl(\sum_{r>0}\frac{1}{r}
\frac{(1-q_2^{-r})(q_1^r+q_3^{r})}
{(1-q_1^{2r})(1-q_3^{2r})}
v^r\Bigr)\,.
\end{align*}
\end{conj}

\section{The Intermediate Long Wave Limit}\label{secILW}

In this section we study the limit $q_1,q_2,q_3\to 1$, 
keeping $r$ and 
$p$
fixed. More specifically, 
with the identification \eqref{q1q2q3} we set 
\[
x^{2r}=e^h\,,\quad x^{2s}=e^{2\tau}\,,\quad \beta=\frac{r-1}{r}\,,
\]
and let $h\to 0$.  Thus, $x\to 1$ and $s\to \infty$ while $r,\tau,\beta$ stay fixed.
We call this limit {\it Intermediate Long Wave (ILW)} limit. 

\subsection{The limit of Hamiltonians} 
In the ILW limit the reduced DVA current \eqref{TDV} becomes
\begin{align*}
&\ssT^{DV}(z)=2+\beta h^2\Bigl(\sum_{n\in\Z}L_nz^{-n}+\frac{(1-\beta)^2}{4\beta}\Bigr)+O(h^4)\,,
\end{align*}
where $L_n$ are the generators of the Virasoro algebra 
with central charge 
and highest weight 
\begin{align}\label{cDelta}
c=1-6\frac{(1-\beta)^2}{\beta}\,,\quad 
\Delta= 
\frac{(1-\beta)^2}{4\beta}(\hat{\pi}^2-1)\,.
\end{align}

The Heisenberg part is a vertex operator, 
$\Delta\,Z(z)=:\exp\bigl(\sum_{m\neq 0}\alpha_m z^{-m}\bigr):$, where
\begin{align*}
&[\alpha_m,\alpha_{n}]=\delta_{m+n}\,\beta\frac{m}{2}h^2
\Bigl(1-m\coth(m\tau)\frac{1}{r}h+O(h^2)\Bigr)\,.
\end{align*}
Rescaling further as 
$\alpha_{-m}=\sqrt{\beta}h\bigl(1+O(h^2)\bigr)a_{-m}$, 
$\alpha_{m}=\sqrt{\beta}h\bigl(1-m\coth(m\tau)h/r+O(h^2)\bigr)a_{m}$ 
($m>0$), 
we have 
$[a_m,a_{n}]=\delta_{m+n,0}m/2$ and $[a_m,L_n]=0$.

Combining these we obtain 
\begin{lem}\label{Litvinov}
In the ILW limit, one obtains from the first non-trivial elliptic local integral of motion 
\begin{align*}
&I_1=2+\beta h^2\mathbf{I}_1+\beta^{3/2}h^3\mathbf{I}_2+O(h^4)\,,\\
\end{align*}
with
\begin{align}
&\mathbf{I}_1=L_0+\frac{(1-\beta)^2}{4\beta}+2\sum_{m>0}a_{-m}a_m\,,
\label{Litv-I1} \\
&\mathbf{I}_2=\sum_{m\neq0}L_{-m}a_m
-\frac{2(1-\beta)}{\sqrt{\beta}}\sum_{m>0}m\coth(m\tau)a_{-m}a_m
+\frac{1}{3}\sum_{k+l+m=0}a_ka_la_m
\,.
\label{Litv-I2} 
\end{align}
\qed
\end{lem}

Operators \eqref{Litv-I1}, \eqref{Litv-I2}
are the first two members of the commuting family 
introduced and termed ILW hierarchy in \cite{L}.

The limit of other integrals is harder to compute, nevertheless the limit of the algebra generated by the local integrals of motion
is a commutative algebra of the same size.  Following \cite{L}, we expect that for generic values of parameters, the operator $\mathbf{I}_2$ is diagonalizable and has simple spectrum. Therefore the higher Hamiltonians can be found by the condition of commuting with 
$\mathbf{I}_2$. In particular, any eigenvector of $\mathbf{I}_2$ is automatically an eigenvector of the whole commutative algebra.

\subsection{The Bethe ansatz}
We study the limit of the solutions of the Bethe ansatz equations 
\eqref{BAE gl1} with $M=2$
in the ILW limit. 

We set $\epsilon=h/r$, $v_i=1+\tilde v_i \epsilon +o(\epsilon)$.
\begin{prop}\label{prop:ILW} 
Let $\tau \neq 0$ and let $t_i$, $i=1,\dots,N$ be a solution of 
\eqref{BAE gl1}. Then
in the 
ILW limit  
we have the following asymptotics: 
$t_i=1+\tilde t_i \epsilon+o(\epsilon)$, as $\epsilon\to 0$ for 
$i=1,\dots, N$. Moreover, the main terms $\tilde t_i$ satisfy 
the following system of equations:
\begin{align}\label{BAE limit}
e^{2\tau} \prod_{s=1}^2\frac{\tilde t_i-\tilde v_s}
{\tilde t_i-\tilde v_s-1}\prod_{j=1}^N 
\frac{\tilde t_i-\tilde t_j-1}{\tilde t_i-\tilde t_j+1}\ 
\frac{\tilde t_i-\tilde t_j+r}{\tilde t_i-\tilde t_j-r}\ 
\frac{\tilde t_i-\tilde t_j-r+1}{\tilde t_i-\tilde t_j+r-1}=-1,
\end{align}
$i=1,\dots,N$.
\end{prop}
\begin{proof}
Suppose for some $a\in\C\cup\{\infty\}$, $a\neq 1$, we 
have a nonempty set $J\subset\{1,\dots,N\}$ such that 
$t_j=a+ \epsilon s_j+o(\epsilon)$ as $\epsilon\to 0$ 
for $j\in J$. Let us assume $\lim_{\epsilon\to 0} t_j\neq a$ 
for $j\not\in J$. Then substituting these asymptotics to the 
Bethe ansatz equations with $j\in J$ and taking the main term 
we obtain the following equations for $s_j$:
\begin{align*}
e^{2\tau} \prod_{j\in J} 
\frac{ s_i- s_j-a}{\tilde s_i- s_j+a}\ \frac{ s_i-s_j+ra}{ s_i- s_j-ra}\ 
\frac{ s_i-s_j-(r-1)a}{s_i- s_j+(r-1)a}=-1\,.
\end{align*}
But this system of equations has no solutions. Indeed, 
multiplying all equations together we obtain 
$e^{|J|\tau}=1$ which is a contradiction, since $\tau\neq 0$ 
and $J$ is nonempty.

It follows that  $\lim_{\epsilon\to 0} t_i=1$ for all $i$. 
The proposition follows.
\end{proof}
Equation \eqref{BAE limit} coincides with (3.1) in \cite{L} 
with identification of parameters 
\begin{align}\label{L par}
\sqrt{-1}b=\sqrt{\beta}\,,\quad 
P^2=\frac{c-1}{24}\hat{\pi}^2\,,
\quad 
\sqrt{-1}x_j=\frac{1-\beta}{\sqrt{\beta}}\tilde{t}_j\,.
\end{align}
We note that $P$ in \eqref{L par} originated in \cite{L} should not be confused with the zero mode operator $P$  in Section \ref{sec FKSW} originated in \cite{FKSW}.

From Corollary \ref{eigvI}  
and  Proposition \ref{prop:ILW},  
we obtain a proof of a conjecture of \cite{L}\footnote{ We believe that there is a sign error in the 
formula (3.2) of \cite{L} for the eigenvalue of  ${\bf I}_2$, namely, the formula  $\mathfrak{h}_2^{(N)}(\tau,P)=-i\sum_{j=1}^N x_j$ 
must read $\mathfrak{h}_2^{(N)}(\tau,P)=i\sum_{j=1}^N x_j$.}.
\begin{cor}\label{ILW spectrum}
For each eigenvector $w$ of $\mathbf{I}_2$ given by \eqref{Litv-I2}, 
there exists a solution $\tilde t_1,\dots, \tilde t_N$ of the equation 
\eqref{BAE limit} such that the eigenvalue of $\mathbf{I}_2$ is  
$\sqrt{\beta}^{-1}(1-\beta)\sum_{j=1}^N\tilde t_j$.
\end{cor}
\begin{proof}
By Corollary \ref{eigvI} and formula \eqref{identity2},
the corresponding eigenvalue of $I_1$ is given by 
\[
I_1= -q(1-q_1)(1-q_3)\sum_{j=1}^Nt_j+q^{-1}(v_1+v_2)\,.
\]
We choose $q^{-1}(v_1+v_2)=2$. Recall that $q_1=e^{-(r-1)\epsilon}$, $q_2=e^{-\epsilon}$, 
$q_3=e^{r\epsilon}$ and $\beta=(r-1)/r$. 
In the limit $\epsilon=h/r \to 0$ we have $t_j=1+\epsilon \tilde{t}_j+\cdots$, so that 
the RHS becomes
\[
 2+2\beta h^2N +\beta^{3/2}h^3\times\frac{1-\beta}{\sqrt{\beta}}
\sum_{j=1}^N\tilde{t}_j+O(h^4)\,.
\]
Comparing this with the expansion in Lemma \ref{Litvinov} we obtain the stated result.

\end{proof}

\section{Quantum toroidal ($\gl_N$,$\gl_M$) duality of XXZ models}\label{sec:duality}
We discuss the quantum toroidal analog of a duality
known for Gaudin models, see  \cite{MTV1}, \cite{MTV2}.

\subsection{The ($\E_1$,$\E_n$) duality}\label{1-1 duality}

Consider an $n$-fold tensor product $\F(u_1)\otimes\cdots\otimes\F(u_n)$ of Fock representations 
of algebra $\E_1=\E_1(q_1,q_2,q_3)$. At the same time consider the Fock representation $\F_{\mu}^\vee(v)$ of $\E_n=\E_n(q_1^\vee,q_2^\vee,q_3^\vee)$. We always assume $q_2^\vee=q_2=q^2$ while $q_1$, $q_1^\vee$ and $u_1,\dots, u_n$, $v$ are arbitrary non-zero parameters.

We identify the $\F(u_1)\otimes\cdots\otimes\F(u_n)$  with the "`sector"' $\F_{\mu,\beta}^\vee=\cH\otimes \C e^{\beta}\subset \F_{\mu}^\vee(u)$ of fixed $\beta\in \bar Q+\bar\Lambda_{\mu}$ by identifying the respective generators
of the Heisenberg algebras as follows.

Let $H_{m}^{\perp,(i)}=1\otimes\cdots\otimes H^{\perp}_m\otimes\cdots\otimes 1$ be the Heisenberg generators for $\E_1$ modules and
$H^{\perp,\vee}_{i,-m}$ for the $\E_n$ module.

For $n\ge3$ we set 
\begin{align*}
&(q_3^{m/2}-q_3^{-m/2})H^{\perp,\vee}_{i,-m}=
(q_3^\vee q_2^{1/2})^{(n-i+1)m}\bigl(H_{-m}^{\perp,(i-1)}
-q_2^{-m/2}
H_{-m}^{\perp,(i)}
\bigr)\,,
\\
&(q_1^{m/2}-q_1^{-m/2})H^{\perp,\vee}_{i,m}=
(q_3^\vee q_2^{1/2})^{-(n-i+1)m}\bigl(q_2^{m/2}H_m^{\perp,(i-1)}-H_m^{\perp,(i)}\bigr)\,,
\end{align*}
where $m>0$ and $H_m^{\perp,(0)}$ is understood as 
$(q_3^\vee q_2^{1/2})^{nm}H_m^{\perp,(n)}$. 
For $n=2$ we modify this as
\begin{align*}
&(q_3^{m/2}-q_3^{-m/2})H^{\perp,\vee}_{0,-m}=
q_2^{m/2}H_{-m}^{\perp,(1)}-H_{-m}^{\perp,(2)}\,,\\
&(q_3^{m/2}-q_3^{-m/2})H^{\perp,\vee}_{1,-m}=
-(q_3^\vee)^m q_2^{m/2}H_{-m}^{\perp,(1)}+(q_3^\vee)^{-m}H_{-m}^{\perp,(2)}\,,\\
&(q_1^{m/2}-q_1^{-m/2})H^{\perp,\vee}_{0,m}=
H_{m}^{\perp,(1)}-q_2^{-m/2} H_{m}^{\perp,(2)}\,,\\
&(q_1^{m/2}-q_1^{-m/2})H^{\perp,\vee}_{1,m}=
-(q_3^\vee)^{-m}q_2^{-m}H_{m}^{\perp,(1)}+(q_3^\vee)^{m}q_2^{m/2}H_{m}^{\perp,(2)}\,,
\end{align*}
and for $n=1$
\begin{align*}
&\frac{H^{\perp,\vee}_{-m}}{1-(q_3^\vee)^{-m}}=-
\frac{H^{\perp}_{-m}}{1-q_3^{-m}}\,,
\quad
\frac{H^{\perp,\vee}_m}{1-(q_1^\vee)^{-m}}=-
\frac{H^{\perp}_m}{1-q_1^{-m}}\,.
\end{align*}
\medskip 

The Taylor coefficients $\{I_N\}$ of $\mc E_1$ transfer matrix depend on the parameters $q_1,q_2,q_3$ of the algebra 
and the twisting parameter
$p=\bar{p} q_2^{-n/2}$.  We also use $p^{*}=\bar{p} q_2^{n/2}$ 
so that $(p^{*})^{-1/n}p^{1/n} q_2=1$.

The Taylor coefficients $\{G_{\nu,N}\}$ of $\mc E_{n}$ transfer matrix  act on each `sector'
$\F_{\mu,\beta}^\vee$.
These coefficients depend on the parameters $q_1^\vee,q_2^\vee,q_3^\vee$ of the algebra and the twisting parameters
\[
p^\vee=\bar{p}^\vee q\,,\quad p_1^\vee=\bar{p}^\vee_1q^{-(\bar\alpha_1,\beta)}\,,\quad
\dots, \quad p_{n-1}^\vee =\bar{p}^\vee_{n-1}
q^{-(\bar\alpha_{n-1},\beta)}
\,.
\]
We also use $p^{*\,\vee}=p^\vee  q_2$.

We match the two sets of parameters as follows.
\begin{align*}
&q_1^\vee=(p^{*})^{-1/n}\,,\quad q_2^\vee=q_2\,,\quad
q_3^\vee=p^{1/n}\,,\quad p^\vee=q_3\,,\quad
p^{*\,\vee}=q_1^{-1}\,,
\\
&p_1^\vee=\frac{u_2}{u_1}\,,\quad \dots,\quad  p_{n-1}^\vee=\frac{u_n}{u_{n-1}}\,.
\end{align*}
The parameters $v$ and $\prod_{i=1}^n u_i$ are not essential for what follows and can be arbitrary.

The $(\E_1,\E_n)$ duality of XXZ systems is then the following statement.
\begin{thm} \label{1n duality}

The $\E_1(q_1,q_2,q_3)$ transfer matrices commute 
with the $\E_n(q_1^\vee,q_2^\vee,q_3^\vee)$ transfer matrices on the sector $\F_{\mu,\beta}^\vee$. 
\end{thm}

This is a re-interpretation of the commutativity between the elliptic local and non-local
integrals of motion, established in \cite{FKSW},\cite{FKSW2}, \cite{KS} by direct computation.
Note that to prove Theorem \ref{1n duality}  it is sufficient to prove commutativity of the first non-trivial integrals only. Indeed, we know that $\E_n$ transfer matrices commute and that the $\E_1$ transfer matrices commute. We also know that the spectrum of the $\E_1$ integral $I_1$ is simple when acting on a generic tensor product of Fock representations. So, the computation is considerably simpler than in \cite{FKSW}, \cite{FKSW2}, \cite{KS}, where the commutativity of all integrals of motions was checked by a tedious direct computation.

Note also that in the non-affine versions of the duality, see for example, \cite{MTV1}-\cite{MTV3},
the twisting parameters of the model on one side correspond to the evaluation parameters on 
the other side.
However, in the affine setting we observe a new feature. Namely, the twisting parameter 
corresponding to the null root is exchanged with the parameter of the dual algebra.

It should be stressed 
that the commutativity concerns only the integrals of motion; 
the quantum toroidal algebras
$\E_1(q_1,q_2,q_3)$ and $\E_n(q^\vee_1,q^\vee_2,q^\vee_3)$
do not commute themselves. 
For instance we have for $n=1$
\begin{align*}
&(w-z)(w-q_1^\vee q_3z)(q_1 w-z)(q_3^\vee w-z) F^{\perp,\vee}(z) F^{\perp}(w)
\\
&=(z-w)(z-q_1 q_3^\vee w)(q_1^\vee z-w)(q_3 z-w) F^{\perp}(w)F^{\perp,\vee}(z)\,.
\end{align*}
 For $n>1$ the relations are more complicated.

\subsection{The general case}\label{m-n duality}\label{duality sec}
In general we conjecture that the following is true.

Let $m,n\in\Z_{\geq 1}$.
Consider the $nm$ bosons $H_s^{(i,j)}$, $s\in\Z\setminus \{0\}$, 
$i=1,\dots, n$, $j=1,\dots, m$,
$[H_s^{(i,j)},H_r^{(k,l)}]=s
\delta_{ik}\delta_{jl}\delta_{s,-r}$.
Let ${\mc U}_{m,n}=\C[H_s^{(i,j)},\ s<0]^{i=1,\dots, n}_{j=1,\dots,m}$. 

Consider the group $G=\Z^{mn}$ and fix free generators $e^{(i,j)}$, $i=1,\dots, n$, $j=1,\dots, m$. 

For any $a\in\Z$, $i\in\{1,\dots, n\}$, 
denote  $\Z^{m-1}_{i,a}=\sum_{j=1}^m a_{j}e^{(i,j)} 
\mid
\sum_{j=1}^m a_{j}=a\}\subset G$. 
Then $\Z^{m-1}_{i,0}$ is a subgroup of $\Z^{mn}$ isomorphic to $\Z^{m-1}$, 
the set $\Z^{m-1}_{i,a}
=ae^{(i,1)}+ \Z^{m-1}_{i,0}$ is a coset,  
and $G=\oplus_{i,a} \Z^{m-1}_{i,a}$. 
  
Similarly, for any $b\in\Z$, $j\in\{1,\dots, m\}$, denote  
$\Z^{n-1}_{j,b}=\sum_{i=1}^n b_{i}e^{(i,j)} 
\mid
\sum_{i=1}^n b_{i}=b\}\subset G$.
	
	Let $\C G$ be the group ring of $G$ and $\C\Z^{m-1}_{i,a}$,  $\C\Z^{n-1}_{j,b}$ the corresponding complexifications.
	
Consider the vector space $V={\mc U_{m,n}}\otimes \C G$.
 Choose $\nu_1,\dots,\nu_n\in\{0,1,\dots,m-1\}$ and  $\mu_1,\dots,\mu_m\in\{0,1,\dots,n-1\}$.

Then we expect that one can define an action of 
quantum toroidal algebras $\mc E_m(q_1,q_2,q_3)$ and $\mc E_n(q_1^\vee,q_2,q_3^\vee)$ on $V$ such that the following is true.

\begin{enumerate}
\item As an $\mc E_m$ module,  $V$ is a direct sum of tensor products of $n$ Fock modules: for any $\bs a=(a_1,\dots, a_n)\in \Z^n$, 
there exist $(u_1(\bs a),\dots,u_n(\bs a))\in(\C^\times)^n$, 
such that
$$V_{n,\bs a}:={\mc U_{m,n}}\otimes\left({\otimes}_{i=1}^n \C \Z^{m-1}_{i,a_i}\right)=\otimes_{i=1}^n \mc F_{\nu_i}(u_i(\bs a)).$$

\item As an $\mc E_n$ module, $V$ is a direct sum of tensor products of $m$ Fock modules: for any $\bs b=(b_1,\dots, b_m)\in \Z^m$, there exist $(v_1(\bs b),\dots,v_m(\bs b))\in(\C^\times)^m$, such that
$$V_{m,\bs b}:={\mc U_{m,n}}\otimes\left({\otimes}_{j=1}^m \C \Z^{n-1}_{j,b_j}\right)=\otimes_{j=1}^m \mc F_{\mu_j}(v_j(\bs b)).$$

\item The $\mc E_n$ transfer  matrices commute with  $\mc E_m$ transfer  matrices. Namely, let $\bs a,\bs b$ be such that $\sum_{i=1}^n a_i=\sum_{j=1}^m b_j$, and consider $V_{\bs a,\bs b}:=V_{n,\bs a}\cap V_{m,\bs b}$. Then the $\mc E_m$ and $\E_n$ transfer matrices act in $V_{\bs a,\bs b}$. Moreover, these transfer matrices commute 
if the twisting parameters $(p_1,\dots,p_{m-1})$ are related to evaluation parameters $(v_1(\bs b):\dots:v_m(\bs b))$ and, similarly 
$(p^\vee_1,\dots,p^\vee_{n-1})$ are related to $(u_1(\bs a):\cdots:u_n(\bs a))$ while the 
 twisting parameters corresponding to the null root, $p$ and $p^\vee$, are related to $q_3^\vee$ and $q_3$.

\end{enumerate}
The products
$\prod_{i=1}^n u_i(\bs a)$
and 
$\prod_{j=1}^m v_j(\bs b)$
are not fixed by the requirement of commutativity of transfer matrices and can be chosen arbitrarily.

\medskip

Theorem \ref{1n duality} asserts that above picture holds for the case of $m=1$ and any $n\in \Z_{\geq 1}$. In this case we have $G=\Z^n$. We can choose in an arbitrary way the following parameters: $q_1$, $q_1^\vee$, $q_2=q_2^\vee$, $a\in\Z$, $v\in\C^\times$, $\nu\in\{0,\dots,n-1\}$, $i\in\{1,\dots,n\}$, and $(u_1,\dots, u_n)\in(\C^\times)^n$.

Then, the space  $\mc U_{1,n}\otimes \C\Z^n_a$ is identified with $\mc E_n$ module $\F_{\mu}(v)$ and the
the space $\mc U_{1,n}\otimes \C(0,\dots,a,\dots,0)$, where $a$ is in the $i$-th position, with $\mc E_1$ module 
$\F(u_1)\otimes\cdots\otimes \F(u_n)$.

We plan to describe the $(\mc E_m,\mc E_n)$ duality in a subsequent paper.

\subsection{The duality between qKdV and quantum affine Gaudin models}
We discuss the specializations of the $(\E_1,\E_2)$ duality of XXZ systems
to ILW and then, further, to conformal limits.

\medskip

Recall the $(\E_1,\E_2)$ duality of XXZ systems discussed in Section \ref{m-n duality}. 
The algebras of local integrals of motion $I_N$ and of non-local integrals of motion 
$G_{\nu, N}$ commute. 
In particular, the Bethe vectors of the two algebras are the same (up to proportionality). 
Therefore, we expect that there is a natural bijection between solutions of 
the $\mc E_1$ 
Bethe ansatz equations \eqref{BAE gl1} with $M=2$
and those of the $\mc E_2$ Bethe ansatz equations \eqref{BAE gl2} 
with $M_0=1$ and $M_1=0$ (where all parameters have checks).

We considered the ILW limit
of the $\mc E_1$ side in Section \ref{secILW}.
For the $\mc E_2$ side we have
\begin{align*}
&q_1^\vee=e^{-\tau}\bigl(1+\epsilon+o(\epsilon)\bigr)\,,\quad 
q_2^\vee=1-\epsilon+o(\epsilon)\,, \quad 
q_3^\vee=e^{\tau}\,,  \\ 
&p^{\vee}
=1+r \epsilon+o(\epsilon)\,,\quad  
p^\vee_1
=1-\frac{\hat{\pi}}{2} \epsilon+o(\epsilon)\,,
\quad v\,,\tau : \text{fixed}\,.
\end{align*}
It is not immediate to take the limit of the non-local integral of 
motions $G_{\nu,N}$. However, we can easily take the limit of the 
$\mc E_2$ 
Bethe ansatz equations.

\medskip
Let $s_i, t_j$ be solutions of \eqref{BAE gl2} with $M_0=1$ and $M_1=0$. 
Let their ILW limits
be  $\lim_{\epsilon\to 0}s_i=\tilde s_i$,
and $\lim_{\epsilon\to 0}t_j=\tilde t_j$, where  $i=1,\dots,N_0$, $j=1,\dots, N_1$.  
Then we easily get the following lemma.

\begin{lem}\label{gl2 ILW}
The limits $\tilde s_i$, $\tilde t_j$ satisfy the system of equations:
\begin{align}
\frac{r-\hat{\pi}-2}{\tilde s_i}+\frac{1}{\tilde s_i-v}
-\sum_{k=1,\ k\neq i}^{N_0}\frac{2}{\tilde s_i-\tilde s_k}
+\sum_{k=1}^{N_1}\left(\frac{1}{\tilde s_i-e^{\tau}\tilde  t_k}
+\frac{1}{\tilde s_i-e^{-\tau} \tilde t_k}\right)=0, 
\label{BAE limit gl2}\\ 
\frac{\hat{\pi}-1}{\tilde t_j}-\sum_{k=1,\ k\neq j}^{N_1}
\frac{2}{\tilde t_j-\tilde t_k}+
\sum_{k=1}^{N_0}\left(\frac{1}{\tilde t_j-e^{\tau} \tilde s_k}
+\frac{1}{\tilde t_j-e^{-\tau} \tilde s_k}\right)=0, 
\label{BAE limit gl2-2}
\end{align}
where $i=1,\dots,N_0$ and $j=1,\dots, N_1$. \qed
\end{lem}

Thus, we expect that there is a natural bijection between solutions of the equations 
\eqref{BAE limit gl2}, \eqref{BAE limit gl2-2}
and those of \eqref{BAE limit}. 
This is a duality of the XXX type equation with 
a Gaudin type equation. In a spirit, it is similar to the duality 
between trigonometric Gaudin $\gl_m$ model and XXX (Yangian related) $\gl_n$ model, 
observed in \cite{MTV3}. 
On the other hand, 
the equations described in Lemma \ref{gl2 ILW} 
do not seem to be in the literature.

\medskip

In the conformal limit we further send $\tau \to 0$.  Let $\lim_{\epsilon\to 0}\tilde s_i=\bar s_i$,
and $\lim_{\epsilon\to 0}\tilde t_j=\bar t_j$  
be the corresponding limits. Then we obviously have the following lemma.
\begin{lem}\label{affine Gaudin}
The limits $\bar s_i$, $\bar t_j$ satisfy the  system of equations:
\begin{align}\label{affine Gaudin BAE}
\frac{r-\hat{\pi}-2}{\bar s_i}+\frac{1}{\bar s_i-v}-\sum_{k=1,\ k\neq i}^{N_0}\frac{2}{\bar s_i-\bar s_k}+\sum_{k=1}^{N_1}\frac{2}{\bar s_i-\bar  t_k}=0, \\ 
\frac{\hat{\pi}-1}{\bar t_j}-\sum_{k=1,\ k\neq j}^{N_1}\frac{2}{\bar t_j-\bar t_k}+\sum_{k=1}^{N_0}\frac{2}{\bar t_j-\bar s_k}=0, \label{affine Gaudin BAE 2}
\end{align}
where $i=1,\dots, N_0$ and $j=1,\dots,N_1$. \qed
\end{lem}
Curiously, these equations coincide with the Bethe ansatz equations 
for the trigonometric Gaudin model associated with the vacuum level 1 
representation $L_{1,0}$ of the affine Lie 
algebra $\hat{\mathfrak{sl}}_2$. 

One can also think of these equations as the Bethe ansatz equations for 
the Gaudin model 
associated with the affine Lie algebra 
$\hat{\mathfrak{sl}}_2$ acting 
on
a tensor product of the basic representation $L_{1,0}$ and of the 
Verma module $V_{r-\hat{\pi}-2,\hat{\pi}-1}$ with highest weight 
$(r-\hat{\pi}-2)\La_0+(\hat{\pi}-1)\La_1$. The Hamiltonians are expected to commute with the diagonal action 
of $\hat{\mathfrak{sl}}_2$ and therefore the model is reduced to the coset theory. 
The spaces of singular vectors are irreducible modules $M^{Vir}_{c,\Delta}$ of 
central charge $c$ and highest weight $\Delta$
over the Virasoro algebra given by the Sugawara construction.

If 
\begin{align*}
L_{1,0}\otimes V_{a,b}= V_{a+1,b}\otimes M^{Vir}_{c,\Delta}+\dots,
\end{align*}
then it is well known that $k=a+b$ and
\begin{align*}
c=1-\frac{6}{(k+2)(k+3)}, \qquad \Delta=\frac{b(b+2)}{4(k+2)(k+3)}.
\end{align*}
Substituting $a=r-\hat{\pi}-2$,  $b=\hat{\pi}-1$, 
we see that the parameters $(c,\Delta)$
of the Virasoro module coincide with \eqref{cDelta}.

In particular, it is natural to expect that the Gaudin model for the 
affine $\widehat{\mathfrak{sl}}_2$ model is dual to the qKdV model. 
The Bethe ansatz equations in Lemma \ref{affine Gaudin} with $N_0=N_1=N$ 
are dual to the Yangian type equations \eqref{BAE limit} with $\tau=0$ 
and to the equations (3) of \cite{BLZ5},  
cf. also 
(2.17) of \cite{L}.
We note that this duality is similar to the known duality between trigonometric Gaudin and XXX models in the non-affine setting, see \cite{MTV3}. Note, that, the level $r-3$ of $\widehat{\mathfrak{sl}}_2$
on the Gaudin side is related to the parameter of the algebra on the XXX side. 

\medskip

We summarize the models discussed above in the following diagram.
\bigskip

$$
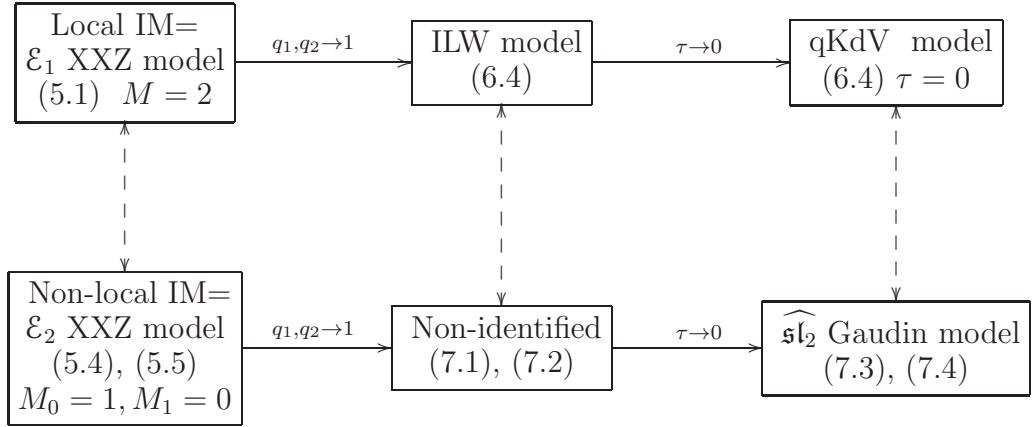
\xymatrix@=20mm{
*+[F]\txt{{\rm Local\ IM}=\\ $\E_1$\ {\rm XXZ\ model}\\
\eqref{BAE gl1} \ $M=2$}\ 
\ar[r]^{\tiny q_1,q_2\to 1}
\ar@{<-->}[d]
&
*+[F]\txt{
{\rm ILW\ model} \\ \eqref{BAE limit}}\
\ar[r]^{\tiny \tau\to 0}
\ar@{<-->}[d]
&
*+[F]\txt{
{\rm qKdV \ model}\\ \eqref{BAE limit}\ 
$\tau=0$}\
\ar@{<-->}[d]
\\
*+[F]\txt{
{\rm Non}\mbox{-}{\rm local\ IM}=\\
$\E_2$\ {\rm XXZ\ model} \\\eqref{BAE gl2}, \eqref{BAE gl2-2}\\
$M_0=1, M_1=0$}\
\ar[r]^{\tiny q_1,q_2\to 1}
&
*+[F]\txt{
{\rm Non\mbox{-}identified}\\ 
 \eqref{BAE limit gl2}, \eqref{BAE limit gl2-2} }\
\ar[r]^{\tiny \tau\to 0}
&
*+[F]\txt{
$\widehat{\mathfrak{sl}_2}$\ {\rm Gaudin\ model}\\
 \eqref{affine Gaudin BAE}, \eqref{affine Gaudin BAE 2}}
\\
}
$$

\bigskip

Here we give the names of the models, refer to corresponding Bethe ansatz equations and indicate dualities between various models
by the vertical dashed lines.

\subsection{Spectrum of non-local integrals of motion of qKdV model} 
According to \cite{FKSW}, it is expected that in the conformal limit the non-local integrals of motions 
$\ssG_{\nu,m}$ turn into non-local integrals of motion for the qKdV model of \cite{BLZ1}. By Corollary \ref{gl2 transfer}, the 
non-local integrals of motions $\ssG_{\nu,m}$ coincide with coefficients of transfer matrices  associated to 
$\E_2$. By Conjecture \ref{gl2 BAE conj}, the spectrum of these transfer matrices is given by Bethe ansatz equations 
\eqref{BAE gl2},\eqref{BAE gl2-2}. 
Therefore we expect that the affine Gaudin Bethe equations \eqref{affine Gaudin BAE}, \eqref{affine Gaudin BAE 2} 
describe the spectrum of non-local integrals of motion for the qKdV model of \cite{BLZ1}.

Recall that a
polynomial $f$ in variables $x_j,y_i$ is called supersymmetric if it is symmetric in variables $x_j$, symmetric in variables $y_i$ and $f|_{x_1=y_1=z}$ does not depend on $z$. Alternatively, 
supersymmetric polynomials are polynomials in power sums
${\rm p}_k=\sum_j{x_j^k}-\sum_i{y_i^k}$, $k\in\Z_{\geq 0}$.  
By definition, a supersymmetric polynomial in infinitely many 
variables is a polynomial in ${\rm p}_k$. The value $f(t,s)$ 
of a supersymmetric polynomial $f$ in infinitely many variables 
at sequences of any size $s=(s_i)_{i=1,\dots,N_0}$, $t=(t_j)_{j=1,\dots,N_1}$ 
is computed by replacing ${\rm p}_k$ with the number 
$\sum_{i=1}^{N_0}{s_i^k}-\sum_{j=1}^{N_1}{t_j^k}$.

Consider conformal limit of $\ab_0$ and $\ab_1$ in Conjecture \ref{gl2 BAE conj}.
Then Taylor coefficients of $u^{-n}$  are supersymmetric polynomials in $s_j, t_i$ of degree $n$.

\medskip 

Consider the non-local integral of motion $\bs G_n$, given by formula (2.16) of \cite{BLZ2} (also (56) in \cite{BLZ1}), acting   
on the
Virasoro Verma module with parameters \eqref{cDelta}. 
Recall also the function $\bs G_1^{(vac)}$  given by (61) in \cite{BLZ1}:
\begin{align}\label{Gvac}
\bs G_1^{(vac)}=\frac{4\pi^2\Gamma(1-2\beta)}
{\Gamma(1-\beta-2P)\Gamma(1-\beta+2P)}\,,
\end{align}
where $P$ is as in \eqref{L par}, and our $\beta$ is $\beta^2$ in \cite{BLZ1}.
\begin{conj}  

There exist
supersymmetric polynomials in infinitely many variables $r_n$ of degree $n$ such that 
every eigenvalue of non-local integral of motion $\bs G_n$ in the subspace of level $N$ has the form $r_n(\bar s,\bar t)$ 
for some solution $\bar s=(\bar s_i)_{i=1,\dots, N}$, $\bar t=(\bar t_j)_{j=1,\dots, N}$ of affine Gaudin Bethe ansatz equations 
\eqref{affine Gaudin BAE}, \eqref{affine Gaudin BAE 2} with $N_0=N_1=N$.

In particular, we have
\begin{align}\label{Q2 spectrum}
r_1(\bar s,\bar t)=
\bs G_1^{(vac)}
\left(1-\frac{1}{v}\,\frac{2(1-2\beta)}{1-\beta+2P}
\sum_{i=1}^N(\bar s_i-\bar t_i)\right).
\end{align}

For generic $r$ and $\hat \pi$ the number of solutions $\bar s_j, \bar t_i$  of 
affine Gaudin Bethe ansatz equations \eqref{affine Gaudin BAE}, \eqref{affine Gaudin BAE 2} with $N_0=N_1=N$   
equals the number of partitions of $N$. 
\end{conj}

\subsection{The BLZ oper}

Let us remark that the Bethe ansatz equations in Lemma \ref{affine Gaudin} for $\bar t_j$ 
can be interpreted as Gaudin equations 
associated to $\mathfrak{sl}_2$ with values in the tensor product of 
Verma module of highest weight $\hat{\pi}-1$ evaluated at $0$,   
and $3$-dimensional modules 
(of highest weight $2$) evaluated at $\bar s_k$, $k=1,\dots, N_0$. 
Then, it is well-known that the eigenvalues $\gamma_i$ of the quadratic 
Gaudin Hamiltonian associated to representation at $\bar s_i$ is
\begin{align*}
\gamma_i=\frac{\hat{\pi}-1}{\bar s_i}
+\sum_{k=1,\ k\neq i}^{N_0}\frac{2}{\bar s_i-\bar s_k}
-\sum_{k=1}^{N_1}\frac{2}{\bar s_i-\bar t_k}\,.
\end{align*}
Therefore the equations for  $\bar s_i$ can be rewritten as
\begin{align}\label{gamma}
\gamma_i=\frac{1}{\bar s_i-v}+\frac{r-3}{\bar s_i}.
\end{align}
Further, to the solution $\bar t_1,\dots,\bar t_{N_1}$
of the $\mathfrak{sl}_2$ 
Gaudin equations, one often associates a scalar differential operator of order 2 which has the form:
\begin{align}\label{oper}
\mc D=\partial_z^2-\frac{l(l+1)}{z^2}+\sum_{i=1}^{N_0}\frac{\gamma_i}{z}
-\sum_{i=1}^{N_0}\left(\frac{2}{(z-\bar s_i)^2}+\frac{\gamma_i}{z-\bar s_i}\right),
\end{align}
where $2l=\hat{\pi}-1$.

Such a differential operator is called an oper. In general, an oper is 
a connection of special form on a principal $G$ bundle defined for each eigenvector of Gaudin system associated to simple Lie algebra $\g$ with corresponding Lie group $G$. In type $A$ opers reduce to scalar differenttial operators.

The oper \eqref{oper} looks tantalizingly similar to the BLZ oper from 
\cite{BLZ5},
see also (2.17) in \cite{L}. Moreover, both opers correspond to the {\it same} eigenvector. 
However, they seem to be fundamentally different in several ways. First, the condition \eqref{gamma} is different from the BLZ condition $\gamma_i z_i=\kappa$. Next, the kernel of \eqref{oper} consists of quasipolynomials (up to a trivial factor), while the BLZ oper has a singularity of order 3 at infinity and transcendental solutions. Finally, the sum $\sum_j \bar s_j$ is related to the eigenvalue of a non-local integral of motion,
see \eqref{Gvac},
whereas the sum $\sum_j z_j$ in the BLZ oper is related to the eigenvalue of a local integral of motion.

Probably, it is fair to say that the BLZ oper and \eqref{oper} are in some sense dual to each other. It is highly desirable to clarify this correspondence.

\appendix

\section{Appendix}

We prove here a combinatorial identity used in the text. 
We use the standard notation $|\la|=\sum_i\la_i$ and
$$
(z)_m=\prod_{s=0}^{m-1}(1-zp^s), \qquad (z)_\infty=\prod_{s=0}^\infty (1-zp^s),
$$
where $m\in\Z_{\geq 0}$.

\begin{prop}\label{identity} We have an equality of formal power series  in variables $p,q_1,q_3$:
\begin{align*}
\sum_{\la\in\mc P} p^{|\la|}
\Bigl(\frac{1}{(1-q_1)(1-q_3)}-\sum_{(i,j)\in\la}q_1^{i-1}q_3^{j-1}\Bigr)
=\frac{(pq_1q_3)_{\infty}}{(q_1)_\infty(q_3)_\infty}.
\end{align*}
\end{prop}
\begin{proof}
The LHS can be rewritten as 
$$
\sum_{a,b\geq 0} q_1^a q_3^b \sum_{\la\in \mc P,\ (a,b)\not\in \la} p^{|\la|}.
$$
The sum over $\la$ is the Poincare series of hook partitions which is easily computed by the Durfee square argument, see \cite{MM}:
$$
\sum_{\la\in \mc P,\ (a,b)\not\in \la} p^{|\la|}
=\sum_{s=0}^{\min\{a,b\}}\frac{p^{(a-s)(b-s)}}{(p)_{a-s}(p)_{b-s}}.
$$
Resumming the LHS one more time we see that it equals to 
$$
\sum_{m,n\geq 0}\frac{p^{mn}}{(p)_m (p)_n}\sum_{s=0}^\infty q_1^{m+s}q_3^{n+s}=\frac{1}{1-q_1q_3}\sum_{m,n\geq 0}\frac{p^{mn}q_1^mq_3^n}{(p)_m (p)_n}.
$$
Using the obvious identity
$$
\sum_{n\geq 0}\frac{z^n}{ (p)_n}=\frac{1}{(z)_\infty}
$$
with $z=p^mq_3$, we, therefore, reduce the identity in the proposition to 
$$
\sum_{m\geq 0}\frac{q_1^m(q_3)_m}{(p)_m}=\frac{(q_1q_3)_\infty}{(q_1)_\infty}\,.
$$
It suffices to check this identity at the 
points $q_3=p^s$, $s=1,2,\dots$.
Then  it reduces to an analog of the Newton binomial formula 
which is readily proved by induction on $s$. The proposition follows.
\end{proof}

{\bf Acknowledgments.}
The research of BF is supported by 
the Russian Science Foundation grant project 16-11-10316. 
MJ is partially supported by 
JSPS KAKENHI Grant Number JP16K05183. 
EM is partially supported by a grant from the Simons Foundation  
\#353831.

EM and BF would like to thank Kyoto University
for hospitality during their visits when this work was started. 
EM would like to thank Rikkyo University for hospitality during his visit while working on this project.

\end{document}